\documentclass[a4paper,12pt]{amsart}
\usepackage{amsmath,amssymb,amsthm,epsfig,epstopdf,url,array}
\usepackage{graphicx}
\usepackage{xcolor}
\usepackage{microtype}
\usepackage{enumitem}
\setlist{nolistsep}
\usepackage[colorlinks=true, allcolors=blue]{hyperref}
\usepackage{csquotes}

\usepackage[a4paper,top=2cm,bottom=2cm,left=1.5cm,right=1.5cm,marginparwidth=1.75cm]{geometry}

\newtheorem{thm}{Theorem}[section]
\newtheorem{prop}[thm]{Proposition}
\newtheorem{cor}[thm]{Corollary}
\newtheorem{lem}[thm]{Lemma}
\newtheorem{conjecture}[thm]{Conjecture}
\newtheorem*{quest}{Question}
\newtheorem*{kc}{Kitaoka's Conjecture}

\theoremstyle{definition}
\newtheorem{remark}[thm]{Remark}
\numberwithin{equation}{section}

\newcommand{\B}{\mathcal{B}}
\newcommand{\ord}{\nu}
\newcommand{\N}{{\mathbb{Z}^+}}
\newcommand{\Z}{{\mathbb{Z}}}
\newcommand{\R}{{\mathbb{R}}}
\newcommand{\C}{{\mathbb{C}}}
\newcommand{\Q}{{\mathbb{Q}}}
\newcommand{\co}{{\mathcal{O}}}
\newcommand{\rank}{\operatorname{rank}}
\newcommand{\Ker}{\operatorname{Ker}}
\newcommand{\cupdot}{\mathbin{\mathaccent\cdot\cup}}
\newcommand{\Gra}[2]{G_{#1 \text{..} #2}}
\newcommand{\Gr}[1]{\Gra{1}{#1}}

\allowdisplaybreaks

\title[Kitaoka's Conjecture for quadratic fields]{Kitaoka's Conjecture for quadratic fields}

\author{V\' \i t\v ezslav Kala}
\address{V. Kala:\newline Charles University, Faculty of Mathematics and Physics, Department of Algebra, Sokolov\-sk\' a 83, 18600 Praha~8, Czech Republic}
\email{vitezslav.kala@matfyz.cuni.cz}

\author{Jakub Kr\'asensk\'y}
\address{J. Kr\'asensk\'y:\newline Czech Technical University in Prague, Faculty of Information Technology, Department of Applied Mathematics, Thákurova~9, 16000 Praha~6, Czech Republic}
\email{jakub.krasensky@fit.cvut.cz}

\author{Dayoon Park}
\address{D. Park:\newline Charles University, Faculty of Mathematics and Physics, Department of Algebra, Sokolov\-sk\' a 83, 18600 Praha~8, Czech Republic}
\email{pdy1016@snu.ac.kr}

\author{Pavlo Yatsyna}
\address{P. Yatsyna:\newline Charles University, Faculty of Mathematics and Physics, Department of Algebra, Sokolov\-sk\' a 83, 18600 Praha~8, Czech Republic}
\email{p.yatsyna@mff.cuni.cz}

\author{B\l{}a\.zej \.Zmija}
\address{B. \.Zmija:\newline Institute of Mathematics of the Polish Academy of Sciences, \'{S}niadeckich 8, 00-656 Warsaw, Poland,}
\address{\hspace{-1.1em}Charles University, Faculty of Mathematics and Physics, Department of Algebra, Sokolov\-sk\' a 83, 18600 Praha~8, Czech Republic}
\email{blazej.zmija@matfyz.cuni.cz}

\thanks{V.K., D.P., and B.\.Z. were supported by {Czech Science Foundation} grant 21-00420M. D.P. and P.Y. were supported by {Charles University} programme UNCE/24/SCI/022. P.Y. was supported by {Charles University} programme PRIMUS/24/SCI/010 and by the Academy of Finland (grant 351271, PI C. Hollanti).
	}
\begin{document}

\begin{abstract} 
We prove that there are at most 13 real quadratic fields that admit a ternary universal quadratic lattice, thus establishing a strong version of Kitaoka's Conjecture for quadratic fields. More generally, we obtain explicit upper bounds on the discriminants of real quadratic fields with a quadratic lattice of rank at most 7 that represents all totally positive multiples of a fixed integer.
\end{abstract}

\maketitle

\tableofcontents

\section{Introduction}

The sets of positive integers that are represented by a positive definite integral quadratic form have long posed exciting research problems. 
The first key invariant is the number of variables of the form: 
A unary form clearly represents $c\sqrt X$ positive integers $<X$, but already the analogous question for binary forms is non-trivial, and it was only in 1912 that Bernays \cite{Be} showed the asymptotic $cX/\sqrt{\log X}$. While a ternary form can never represent all positive integers for local reasons, it follows from the 1990 theorem of Duke--Schulze-Pillot \cite{DSP} that it nevertheless always represents a positive proportion of integers.
Finally, forms in 4 and more variables can quite easily be universal (i.e., represent all positive integers), as illustrated already by the sum of 4 squares.
Building on works of Ramanujan, Dickson, and Willerding (among many others), all universal forms are now classified: First, in 1993 Conway--Schneeberger covered the classical case (i.e., all off-diagonal coefficients being even) as the main part of the proof of the 15-Theorem, see \cite{Bh}, and then
in 2011 Bhargava--Hanke \cite{BH} managed the difficult extension to all forms in the 290-Theorem.

The interest in integral quadratic forms over number fields goes back to the 1920s when Kirmse \cite{Kir} and G\" otzky \cite{Go} studied the sum of 4 squares over the rings of integers of real quadratic fields $\Q(\!\sqrt D)$ with small $D$, especially $D=5$.
In the 1940s, Maa{\ss} \cite{M} showed that actually the sum of 3 squares is universal over $\Q(\!\sqrt5)$, before Siegel \cite{Si} proved that $\Q$ and $\Q(\!\sqrt5)$ are the only totally real number fields with universal sum of any number of squares. Note that, given a totally real number field $F$ with ring of integers $\co_F$, we here call a quadratic form $Q$ with $\co_F$-coefficients universal over $F$ if it is totally positive definite (i.e., attains only totally positive values, except for $Q(0,\dots,0)=0$) and it represents over $\co_F$ all the totally positive elements of $\co_F$. Further, it is common to consider the slightly more general setting of quadratic lattices, as we will do in the rest of the article.

Since then, numerous researchers significantly advanced our understanding of universal lattices over number fields. In particular, we now know that they often require many variables, especially over real quadratic fields \cite{BK1, CKR, Ka1, KKP, KYZ}, but also in higher degrees \cite{Ka3, KT, KTZ, Man, Ya}. For further relevant results, see, e.g., \cite{Ea,CO,KL,KY1}, works on indefinite forms \cite{HHX,XZ}, or the surveys \cite{Ka2,Ki}.

Despite all this progress, it is curious that the most influential conjecture in the area remains open.

\begin{kc}[1990's]
    There are only finitely many totally real number fields that admit a ternary universal quadratic lattice.
\end{kc}

However, there has been some partial progress. Chan--Kim--Raghavan \cite{CKR} solved the case of classical forms over real quadratic fields $\Q(\!\sqrt D)$, where there is a ternary universal form exactly for $D=2,3,5$. In 2022, Kim--Kim--Park \cite{KKP} then obtained significant results for non-classical universal lattices, as they proved that there are only finitely many values of $D$ such that $\Q(\!\sqrt D)$ admits a universal lattice of rank at most 7. In particular, this established Kitaoka's Conjecture for real quadratic fields. However, while their methods give explicit upper bounds on $D$, they are probably too large to enable checking the remaining fields computationally in order to obtain a sharp result. 

In odd degrees $[F:\Q]$, Earnest--Khosravani \cite{EK} showed that there are no universal ternary lattices for local reasons.
Even degrees 4 and higher have not been fully resolved yet, although Kr\' asensk\' y--Tinkov\' a--Zemkov\' a \cite{KTZ} proved that there are no biquadratic fields with a classical universal ternary, and Man \cite{Man} showed that multiquadratic fields admitting a universal ternary have density zero among all multiquadratic fields of fixed degree.

Also, there was a proof of the \textit{weak Kitaoka's Conjecture} that asserts, in each fixed degree $[F:\Q]$, the finiteness of fields with a universal ternary: This was established first by Kim \cite{kim-unp} for classical lattices, and then by Kala--Yatsyna \cite{KY} in general.

\medskip

Our main result is the explicit resolution of Kitaoka's Conjecture for universal lattices over real quadratic fields (without the classical assumption).

\begin{thm}\label{ThmMain2}
If a real quadratic field $\Q(\!\sqrt{D})$ admits a universal ternary lattice, then 
\[
D=2,3,5,6,7,10,13,17,21,33,41,65,77.
\]
\end{thm}

Further, in each of these $13$ cases, we found a candidate for the universal ternary lattice, and so we expect (Conjecture \ref{conj:list}) 
that Theorem \ref{ThmMain2} actually holds as an equivalence.
See Section \ref{sec:ternary} for more details and evidence, as well as for the proof 
of Theorem \ref{ThmMain2} as Theorem \ref{thm:ternary}.

The final part of the proof involves non-trivial computations performed in Magma \cite{magma}. To restrict ourselves to a finite list of fields to check we use our main Theorem \ref{MAINTHEOREM1} that will be discussed below. We also find all the conjectural universal ternaries for all $D > 5$ -- see Theorem \ref{thm:list} and the surrounding discussion. (It is possible to also cover the cases of $D=2,3,5$, but the resulting list of lattices would be quite long and would have questionable value.)

Notably, we also resolve several informal conjectures on separations between minimal ranks of universal forms with various restrictions (Remark \ref{rem1}). Namely, we show that over $\Q(\!\sqrt{13})$, the minimal rank of 
\begin{itemize}
    \item a (non-classical) universal form is 3, 
    \item a classical universal form is 4, 
    \item a diagonal universal form is at least 5. 
\end{itemize}
Further, conjecturally, over $\Q(\!\sqrt{10})$ and $\Q(\!\sqrt{65})$, there is a universal ternary lattice, but there is no universal ternary form, see Remark \ref{rem2}.

\medskip

As the theoretical basis for the computations proving Theorem \ref{ThmMain2}, we obtain more general lower bounds on ranks of lattices which are not necessarily universal, but which represent all totally positive multiples of a fixed rational integer $m$. Note that this covers the seemingly more general case of multiples of some $\alpha\in\co_F^+$, for $\alpha\co_F^+\supset \mathrm{N}_{F/\Q}(\alpha)\co_F^+$.

The consideration of lattices representing $m\co_F^+$ plays two roles. First of all, it is often more convenient to work with classical lattices, even when proving theorems that include the non-classical case. In order to do so, one can just replace a non-classical universal lattice $(L,Q)$ with $(L,2Q)$ -- that is now classical and represents all of $2\co_F^+$. Second, lattices representing $m\co_F^+$ serve as a very natural, non-trivial generalization of universal lattices that was already investigated, e.g., for quaternary lattices \cite{EK} and for sums of squares \cite{KY2,Ras,Chau}.
Such results can hopefully teach us insights which may be further applicable to more general considerations of lattices representing other positive density subsets of $\co_F^+$. Of particular note is the recent work of Bordignon--Cherubini \cite{BC} over $\Z$.

Our finiteness result builds on the works of B.~M.~Kim \cite{Ki1,Ki2} who established that 8 is the smallest rank for which there are infinitely many real quadratic fields admitting a diagonal universal form. As we already mentioned, Kim--Kim--Park \cite{KKP} recently extended these results to show that even when one considers universal quadratic lattices (not necessarily diagonal or classical), there are still only finitely many real quadratic fields with a universal lattice in at most 7 variables. As our other main result, we show the following generalization.

\begin{thm}\label{MAINTHEOREM1} 
Let $m, D$ be positive rational integers, $D \neq 1$ squarefree. Denote $F=\Q(\!\sqrt{D})$. Let $L$ be a classical totally positive definite quadratic $\co_F$-lattice representing every element of $m\co_F^{+}$. Then we have the following bounds on the discriminant of $F$.

\begin{center}
\begin{tabular}{ | c | l | l | }
\hline
  &\ \  for $m =1$ &\ \  for $m=2$ \\ 
\hline
 if $\rank L \leq 3$  &\ \  $\Delta_{D} < 625$ 
& \ \ $\Delta_{D} < 10{,}000$ \\ 
\hline 
 if $\rank L \leq 4$ &\ \  $\Delta_{D} < 2{,}025$ &\ \  $\Delta_{D} < 32{,}400$ \\
\hline
if $\rank L \leq 5$ &\ \  $\Delta_{D} < 50{,}625$ &\ \  $\Delta_{D} < 3{,}240{,}000$ \\
\hline
if $\rank L \leq 6$ &\ \  $\Delta_{D} < 275{,}625$ &\ \  $\Delta_{D} < 2{,}470{,}090{,}000$ \\
\hline
\ \ if $\rank L \leq 7$ \ \ & \ \ $\Delta_{D} < 441{,}000{,}000$ \ \ & \ \ $\Delta_{D} < 63{,}234{,}304{,}000{,}000$ \ \ \\
\hline
\end{tabular}
\end{center}
Further, for every $\varepsilon >0$,
\begin{center}
\begin{tabular}{ | l | l | l | l | }
\hline
  & unconditionally, for $m \geq 3$ & unconditionally & under GRH, for $m\geq 3$ \\ 
\hline
if $\rank L \leq 3$ & $\Delta_{D} < 5^4m^5(\log m)^{2}$  & $\Delta_{D}\ll_{\varepsilon} m^{4+\kappa + \varepsilon}$ & $\Delta_{D} < 2^2\cdot 3^4\cdot 5^2m^{4}(\log m)^{4}$ \\ 
\hline 
 if $\rank L \leq 4$ & $\Delta_{D} < 3^{4}\cdot 5^2 m^5(\log m)^{2}$ & $\Delta_{D}\ll_{\varepsilon} m^{4+\kappa + \varepsilon}$ & $\Delta_{D} < 2^2\cdot 3^8 m^{4}(\log m)^{4}$  \\
\hline 
if $\rank L \leq 5$ & $\Delta_{D} < 3^4\cdot 5^4 m^{7} (\log m)^{2}$ & $\Delta_{D}\ll_{\varepsilon} m^{6+\kappa + \varepsilon}$ & $\Delta_{D} < 2^2\cdot3^8\cdot5^2 m^{6}(\log m)^{4}$ \\
\hline
if $\rank L \leq 6$ & $\Delta_{D} < 2^{6}\cdot 5^{8} m^{55/4}(\log m)^{13/2}$ & $\Delta_{D}\ll_{\varepsilon} m^{11+\kappa+\lambda+\varepsilon}$ & $\Delta_{D}\ll m^{10} (\log m)^{12}$ \\
\hline
if $\rank L \leq 7$ & $\Delta_{D} < 2^{12}\cdot 5^{10} m^{43/2} (\log m)^{9}$ & $\Delta_{D}\ll_{\varepsilon} m^{18+2\lambda +\varepsilon} $ & $\Delta_{D}\ll m^{16}(\log m)^{16}$ \\
\hline
\end{tabular}
\end{center}
where $\kappa=\frac{1}{\sqrt{\mathrm{e}}}\approx 0.6065$ and $\lambda=\frac{\kappa^{2}}{4}+\frac{5\kappa}{2}-1\approx 0.6083$.

\end{thm}

Here and in the following, we use the common analytic number theory notation that $f\ll g$ if there is a constant $c>0$ such that $f(x)<cg(x)$ for all sufficiently large values $x$. The subscript $\ll_{\varepsilon}$ is used to stress that the constant $c$ may depend on the parameter ${\varepsilon}$. In Section \ref{sec:m-uni}, we prove Theorem \ref{MAINTHEOREM1} as a consequence of the more precise statement of Theorem \ref{ThmMain1} that is, in turn, proved in Sections \ref{sec:45} and \ref{sec:678}.

The main idea of the proof is not surprising, as we consider the representability of certain elements. We start by representing $m$ and $2m$, but then we consider a third multiple $C_1m$ where we show that $C_1$ can be chosen in terms of the least quadratic non-residue $\gamma_p$ modulo a suitable prime $p$ (Corollary \ref{ind cor}). Establishing this requires careful local considerations in Section \ref{sec:tech}.

The explicit unconditional bound of Theorem \ref{MAINTHEOREM1} is obtained by estimating $\gamma_p$ using the result by Trevi\~{n}o \cite{Tre}, whereas the two other estimates follow by using Burgess' bound \cite{Bu} or the conditional bounds of Ankeny and Lamzouri--Li--Soundararajan \cite{An,LLS}, which assume the Generalized Riemann Hypothesis (GRH). As one can expect, the latter bounds give better asymptotics in terms of the dependence on $m$, at the cost of having inexplicit constants. In particular, for Theorem \ref{ThmMain2} we need the explicit unconditional bound.

Note that some of the bounds given in Theorem \ref{MAINTHEOREM1} have been improved computationally: Chan--Kim--Raghavan \cite{CKR} (together with our Theorem \ref{thm:list} to cover also non-free lattices) showed that if $\rank L \leq 3$ and $m=1$, then actually $\Delta_D\leq 12$. The bounds for $\Delta_D$ given in Theorem \ref{MAINTHEOREM1} are quite computationally tractable also for other small values of $m$ and bounds on the rank -- this is an interesting goal for a further research project (of highly computational nature).

\medskip

To conclude, we note that no classical universal ternary lattices are known in degrees $>2$. 
Together with the partial evidence given by the weak Kitaoka theorem \cite{KY}, 
this might tempt one to formulate a \textit{strong Kitaoka's Conjecture} that $\Q(\!\sqrt2)$, $\Q(\!\sqrt3)$ and $\Q(\!\sqrt5)$, found in \cite{CKR}, are the only number fields with a classical universal ternary lattice, regardless of the degree. On the other hand, it is somewhat surprising that there are quite a lot of real quadratic fields with a non-classical universal ternary (10 on top of the classical 3 found by \cite{CKR}), and  Krásenský--Scharlau [in preparation]  constructed non-classical universal ternary forms over the quartic fields $\Q(\!\sqrt2,\sqrt5)$ and $\Q(\zeta_{20})^+=\Q\Bigl(\!\!\sqrt{\frac{5+\sqrt5}{2}}\Bigr)$.

\section{Preliminaries}\label{sec:prel}
	
Let us start with general notation. We write $\lfloor \,\cdot\, \rfloor$ and $\lceil \,\cdot\, \rceil$ for the floor and ceiling function, respectively, and $\log$ for the natural logarithm. Throughout the paper, if $p$ is an odd prime, we write $\gamma_p$ for the least quadratic non-residue modulo $p$, and we put $\gamma_2 = 7$. If $p$ is a prime, we use $\ord_p$ for the additive $p$-valuation, and $\Q_p$ and $\Z_p$ for the $p$-adic field and $p$-adic integers. For any ring $R$, the symbol $R^{\times}$ denotes its group of invertible elements. By a notation such as $(\Z_p^{\times})^2$, we mean the set of squares of $p$-adic units; this collides with a notation such as $\R^n$ for the $n$-th Cartesian power, but the meaning will always be clear from the context.

The symbol $M_n(R)$ stands for the set of $n \times n$ matrices with entries from $R$. The matrix can be given directly by its entries in the form $M = (m_{ij})_{1 \leq i,j \leq n}$. We say that $M \in M_n(\R)$ is \emph{positive semidefinite} if $M=M^{\mathrm{T}}$ and $x^{\mathrm{T}}Mx \geq 0$ for every $x \in \R^n$. It is \emph{positive definite} if equality only occurs for the zero vector $x=0$. We will need the following bound for determinants of positive semidefinite matrices.

\begin{lem}[Hadamard's inequality] \label{Hadamard}
Let $A = (a_{ij})_{1\leq i,j \leq n} \in M_n(\R)$ be a positive semidefinite matrix. Then $\det A \leq a_{11}a_{22} \cdots a_{nn}$.
\end{lem}
\begin{proof}
See, for example, \cite[Theorem 7.8.1]{HJ}.
\end{proof}

\subsection{Number fields}

Let $F$ be a number field with ring of integers $\co_F$ and degree $d=[F:\Q]$. Let $\sigma_1 = \mathrm{id}, \sigma_2, \ldots, \sigma_d$ be the embeddings $F \to \C$. We say that $F$ is totally real if $\sigma_i(F) \subset \R$ for every $i$. In such a case, we call $\alpha \in F$ \emph{totally positive} if $\sigma_i(\alpha) > 0$ for every $i$; this is denoted $\alpha \succ 0$. If $\alpha - \beta \succ 0$, we write $\alpha \succ \beta$. Analogously, we write $\alpha \succeq 0$ if $\sigma_i(\alpha) \geq 0$ for every $i$, and $\alpha \succeq \beta$ if $\alpha - \beta \succeq 0$. The set of all totally positive elements of $\co_F$ is denoted $\co_F^+$. In particular, $\Z^+$ is the set of all positive rational integers. A symmetric matrix $M \in M_n(F)$ is \emph{totally positive (semi)definite} if $\sigma_i(M)$ is positive (semi)definite in every embedding $\sigma_i$; this can also be defined by requiring $x^{\mathrm{T}}Mx \succ 0$ (or $x^{\mathrm{T}}Mx \succeq 0$) for every nonzero $x \in F^n$.

For the vast majority of the paper, we shall work with real quadratic fields. If we write $F = \Q(\!\sqrt{D})$, we automatically expect that $D>1$ is squarefree, even if we do not explicitly say so. For $\alpha \in F$, its conjugate is denoted by $\alpha'$; thus, $\alpha \succ 0$ if and only if $\alpha, \alpha' > 0$. We shall fix the following notation:
\[
\Delta_D=\begin{cases}
D & \text{if }D \equiv 1 \pmod 4,\\
4D & \text{if }D \equiv 2,3 \pmod 4
\end{cases}
\]
for the discriminant of $F$, and
\[
\omega_D=\begin{cases}
\frac{1+\sqrt{D}}{2} & \text{if }D \equiv 1 \pmod 4,\\
\sqrt{D} & \text{if }D \equiv 2,3 \pmod 4,
\end{cases}
\]
so that $\co_F = \Z[\omega_D]$. The identity $\omega_D - \omega_D' = \sqrt{\Delta_D}$ will be often useful.

The \emph{norm} of $\alpha \in F$ is $\mathrm{N}_{F/\Q}(\alpha) = \alpha\cdot\alpha'$; an important fact is that up to multiplication by $(\co_F^{\times})^2$, there are only finitely many elements of $\co_F^+$ with norm under a given bound.

\subsection{Quadratic spaces}

Rather than viewing quadratic forms as homogeneous polynomials of degree two, we use the language of quadratic spaces and lattices over number fields, for which \cite{OM} is the standard reference.

Let $F$ be a field, $\mathrm{char}(F) \neq 2$. A \emph{quadratic space} over $F$ is a pair $(V,Q)$, where $V$ is a finite-dimensional vector space over $F$ and $Q : V \to F$ is a \emph{quadratic map}. This means that $Q$ satisfies $Q(\alpha v) = \alpha^2 Q(v)$ for every $\alpha \in F$, $v \in V$, and that the induced map $\B : V \times V \to F$, $\B(v,w) := \frac12 \bigl( Q(v+w) - Q(v) - Q(w) \bigr)$ is bilinear.

The quadratic space $(V,Q)$ \emph{represents} $\alpha \in F$ if there is $v \in V$ such that $Q(v) = \alpha$. The space is \emph{non-degenerate} if there is no $v \in V$ such that $\B(v,w) = 0$ for all $w \in V$. In the following, we assume all quadratic spaces to be non-degenerate. Such a space is \emph{isotropic} if $Q(v) = 0$ for a nonzero $v \in V$, and \emph{anisotropic} otherwise.

Vectors $v_1, \ldots, v_k$ are \emph{pairwise orthogonal} if $\B(v_i,v_j)=0$ for every $i \neq j$. If, moreover, none of them is zero, we can write $F v_1 \perp F v_2 \perp \cdots \perp F v_k$ instead of $Fv_1 + Fv_2 + \cdots + F v_k$ to stress that the subspace is generated by orthogonal vectors.

For vectors $v_1, \ldots, v_k$, their \emph{Gram matrix} is the symmetric $k \times k$ matrix $\bigl(\B(v_i,v_j)\bigr)_{1\leq i,j \leq k}$. An important fact is that the Gram matrices of two different bases of $V$ differ just by multiplication by an element of $(F^{\times})^2$. To increase readability, we introduce the following notation: Whenever vectors $v_1, \ldots, v_k$ are given, we write $\Gr{k}$ for their Gram matrix. More generally, $\Gra{\ell}{k} = \bigl(\B(v_i,v_j)\bigr)_{\ell \leq i,j \leq k}$.

If $F$ is a totally real number field, then we say that $(V,Q)$ is \emph{totally positive (semi)definite} if $Q(x) \succ 0$ (or $Q(x) \succeq 0$) for every nonzero $x \in V$. In a positive semidefinite quadratic space, the Gram matrix of given vectors is always totally positive semidefinite. In particular, the Cauchy--Schwarz inequality $Q(v)Q(w) \succeq \B(v,w)^2$ holds for every $v,w$.

\subsection{Quadratic lattices}

Let $(V,Q)$ be a quadratic space over a number field $F$ and $L$ a finitely generated $\co_F$-submodule of $V$ such that $Q(L) \subset \co_F$. By slight abuse of notation, write $Q: L \to \co_F$ also for the restriction of the quadratic map to $L$. Then we call $(L,Q)$ a \emph{quadratic $\co_F$-lattice}; we often omit the adjective \enquote{quadratic}. Note that our definition corresponds to what many authors call an integral lattice.

We can also directly define a quadratic lattice without using the quadratic space: $L$ is a finitely generated torsion-free $\co_F$-module and $Q: L \to \co_F$ is a quadratic map, just as above. The quadratic space can be recovered by taking the tensor product $L \otimes F$. The rank of a lattice $(L,Q)$, denoted $\rank L $, is the dimension of the space $L \otimes F$. Specifically, quadratic lattices of rank $3$ are called \emph{ternary}.

We often just write $L$ instead of $(L,Q)$ for the quadratic lattice; throughout the paper, we automatically assume that the quadratic map on $L$ is denoted $Q$ (unless otherwise stated), and that $\B$ is the corresponding bilinear map. A quadratic lattice is \emph{classical} if $\B(L,L) \subset \co_F$; for a general non-classical lattice we only have $\B: L \times L \to \frac12 \co_F$. Clearly, a Gram matrix of any $k$ vectors from a classical $\co_F$-lattice lies in $M_k(\co_F)$.

A quadratic $\co_F$-lattice is \emph{totally positive (semi)definite} if the corresponding quadratic space is. Again, $\alpha \in \co_F$ is \emph{represented} by $L$ if $Q(v) = \alpha$ for some $v \in L$. Finally, a quadratic lattice is \emph{universal} if $Q\bigl(L \setminus \{0\}\bigr) = \co_F^+$; i.e., it is totally positive definite and represents all elements of $\co_F^+$.

A quadratic lattice $L$ is \emph{free} if the underlying $\co_F$-module is. In this case, $L = \co_F b_1 + \cdots + \co_F b_n$ for some linearly independent vectors $b_1, \ldots, b_n$, where $n= \rank L$. The vectors $b_i$ form a \emph{basis} of $L$.

Free quadratic lattices correspond to integral quadratic forms: If $\varphi \in \co_F[X_1, \ldots, X_n]$ is a \emph{quadratic form}, i.e., a homogeneous polynomial of degree $2$, then we can use it to define a quadratic lattice $(L,Q)$ as follows: We let $L = \co_F^n$ and let $Q\bigl((\alpha_1,\ldots, \alpha_n)\bigr) = \varphi(\alpha_1,\ldots,\alpha_n)$. On the other hand, given any free quadratic lattice $(L,Q)$ and its basis $b_1, \ldots, b_n$, we can define a quadratic form $\varphi$ by putting
\[
\varphi(X_1, \ldots, X_n) = Q(X_1b_1 + \cdots + X_nb_n).
\]

For a field $F$ with class number $h(F)=1$, every $\co_F$-lattice is free. Generally, by the structure theorem for finitely generated modules over Dedekind domains, every $\co_F$-lattice $L$ can be written as
\[
L = \co_F b_1 + \cdots + \co_F b_{n-1} + \mathfrak{a}^{-1} b_n
\]
for some linearly independent $b_1, \ldots, b_n \in L$ and an ideal $\mathfrak{a}$ in $\co_F$; again, $n = \rank L$. Such a lattice corresponds to a quadratic form where we can plug in $X_1, \ldots, X_{n-1} \in \co_F$, but $X_n$ can take values from the fractional ideal $\mathfrak{a}^{-1}$. We shall meet two forms corresponding to non-free lattices, $\varphi^{(10)}$ and $\varphi^{(65)}$, in Section \ref{subsec:candidates}.

\subsection{Lemmas about \texorpdfstring{$p$}{p}-adic numbers}

We shall need the following well-known results about squares in $\Z_p$. Of course, unless $p=2$, we can make the replacement $4p\Z_p=p\Z_p$.

\begin{lem} \label{local value lem}
Let $p$ be any prime.
\begin{itemize}
\item[(1)] $(1+p\Z_p)^2=1+4p\Z_p$.
\item[(2)] $(u_p+p\Z_p)^2=u_p^2+4p\Z_p$ for $u_p \in \Z_p^{\times}$.
\end{itemize}
\end{lem}
\begin{proof}
See \cite[63:1]{OM}.
\end{proof}

\begin{cor} \label{local value}
Let $p$ denote any odd prime and $\gamma_p$ any quadratic non-residue modulo $p$.
\begin{itemize}
\item[(1)] $\Z_2^{\times}=(\Z_2^{\times})^2 \cupdot 3(\Z_2^{\times})^2 \cupdot 5(\Z_2^{\times})^2 \cupdot 7(\Z_2^{\times})^2$,
\item[(2)] $\Z_p^{\times}=(\Z_p^{\times})^2 \cupdot \gamma_p (\Z_p^{\times})^2$,
\item[(3)] $\Q_2^{\times}=(\Q_2^{\times})^2 \cupdot 3(\Q_2^{\times})^2 \cupdot 5(\Q_2^{\times})^2 \cupdot 7(\Q_2^{\times})^2 \cupdot 2(\Q_2^{\times})^2 \cupdot 2\cdot3(\Q_2^{\times})^2 \cupdot 2\cdot5(\Q_2^{\times})^2 \cupdot 2\cdot7(\Q_2^{\times})^2$,
\item[(4)] $\Q_p^{\times}= (\Q_p^{\times})^2 \cupdot \gamma_p(\Q_p^{\times})^2 \cupdot p(\Q_p^{\times})^2 \cupdot p\gamma_p(\Q_p^{\times})^2$.
\end{itemize}
Here $\cupdot$ denotes a union of disjoint sets.
\end{cor}
\begin{proof}
Follows from \cite[63:1]{OM}.
\end{proof}

For a $\Z$-lattice $L$ and a prime $p$, $L \otimes \Q_p$ is a well-defined quadratic space over $\Q_p$. We need a result about ternary quadratic spaces over $\Q_p$.

\begin{lem}\label{aniso ternary}
Let $L$ be a positive definite ternary $\Z$-lattice. Then the quadratic space $L \otimes \Q_p$ is isotropic if and only if $L\otimes \Q_p$ represents all of $\Q_p$.
\end{lem}
\begin{proof}
    See \cite[Theorem 2.3(2)]{HHX}.
\end{proof}

\section{A more precise version of the main Theorem \ref{MAINTHEOREM1} }\label{sec:m-uni}

In this section, we will state the following result which is a more subtle version of Theorem \ref{MAINTHEOREM1}, and then use it to deduce Theorem \ref{MAINTHEOREM1} itself. Recall that throughout the paper, we write $\gamma_{p}$ for the least positive  quadratic non-residue modulo $p$ if $p$ is an odd prime, and we put $\gamma_{2}=7$.

\begin{thm}\label{ThmMain1} 
Let $m, D\in\N$, $D \neq 1$ squarefree, and $F=\Q(\!\sqrt{D})$. Let $L$ be a classical totally positive definite quadratic $\co_F$-lattice which represents every element of $m\co_F^{+}$. Then there exist primes $p_{1},p_{2}\leq 2m^{2}$, integers $2 \leq C_{i}\leq \gamma_{p_{i}}$, primes $q_{1},q_{2}$ such that $q_{i}\leq 2C_{i}m^{3}$ and integers $1 \leq D_{i}\leq q_{i}\gamma_{q_{i}}$ with the following properties:
\begin{enumerate}
\item if $\rank L\leq 3$, then $\Delta_{D}\leq 4C_{1}m^{2}\big(\omega_{D}-\lfloor \omega_{D}'\rfloor\big)$;
\item if $\rank L \leq 4$, then $\Delta_{D} \leq 4C_{1}m^{2} \bigl(2\omega_{D}-\lfloor 2\omega_{D}'\rfloor\bigr)$;
\item if $\rank L\leq 5$, then
\begin{itemize}
\item $\Delta_{D}\leq  4C_{1}m^{2} \bigl(C_{2}\omega_{D} - \lfloor C_{2}\omega_{D}'\rfloor\bigr)$ or
\item $\sqrt{\Delta_D}^3 \leq 36C_{2}m^{3} \sqrt{\Delta_{D}}^{2} + 18C_2m^{3} \sqrt{\Delta_{D}} + m^{3}$; 
\end{itemize}
\item if $\rank L \leq 6$, then
\begin{itemize}
\item $\Delta_{D} \leq 4C_{1}m^{2} \bigl(C_{2}\omega_{D} - \lfloor C_{2}\omega_{D}'\rfloor\bigr)$ or
\item $\sqrt{\Delta_D}^3 \leq 36C_{2}m^{3} \sqrt{\Delta_{D}}^{2} + 18C_2m^{3} \sqrt{\Delta_{D}} + m^{3}$ or
\item $\Delta_{D} \leq 4D_{1}m^{2} \max\bigl\{C_1,C_{2}\omega_{D} - \lfloor C_{2}\omega_{D}'\rfloor\bigr\}$;
\end{itemize}
\item if $\rank L  \leq 7$, then
\begin{itemize}
\item $\Delta_{D}\leq 4C_{1}D_{1} m^{2}$ or
\item $\Delta_{D} \leq 4\max\{C_1,D_1\} m^{2} \max\bigl\{C_{2}\omega_{D} - \lfloor C_{2}\omega_{D}'\rfloor, D_{2}\omega_{D} - \lfloor D_{2}\omega_{D}'\rfloor \bigr\}$ or
\item $\sqrt{\Delta_D}^3 \leq 36C_{2}m^{3} \sqrt{\Delta_{D}}^{2} + 18C_2m^{3} \sqrt{\Delta_{D}} + m^{3}$ or
\item $\sqrt{\Delta_{D}}^{4} \leq 192C_{2}D_{2}m^{4} \sqrt{\Delta_{D}}^{3} + 144C_{2}D_{2}m^{4}\sqrt{\Delta_{D}}^{2} + 96\max\{C_{2},D_{2}\}m^{4}\sqrt{\Delta_{D}} + m^{4}$.
\end{itemize}
\end{enumerate}
\end{thm}

Note that all the parameters $p_i,q_i,C_i,D_i$ above depend on the lattice $L$. However, their ranges are absolutely bounded.

Sections \ref{sec:tech}, \ref{sec:45}, and \ref{sec:678} are spent by proving this theorem. Here we are going to simplify it. Many of the inequalities can be replaced thanks to the following lemma:

\begin{lem}\label{le:ineqSlayer}
Let $C,m,t \in \N$. If $\Delta_D \leq 4Cm^2 (t\omega_D - \lfloor t\omega_{D}'\rfloor)$, then $\sqrt{\Delta_D} < 4Cm^2 t + 1$.
\end{lem}
\begin{proof}
We have
\[
\Delta_{D}\leq 4Cm^{2}\left(t\omega_{D}-\lfloor t\omega_{D}'\rfloor\right) < 4Cm^{2}\left(t\omega_{D}-t\omega_{D}' +1\right) =4Cm^{2}\bigl(t\sqrt{\Delta_{D}} +1\bigr)\leq 4Cm^{2}t \bigl(\sqrt{\Delta_{D}}+1 \bigr).
\]
Thus $4Cm^{2}t\geq \frac{\Delta_{D}}{\sqrt{\Delta_{D}}+1}>\sqrt{\Delta_D}-1$, just as we needed.
\end{proof}

Now, as the first step in the proof of Theorem \ref{MAINTHEOREM1}, we present the following simplified version of Theorem \ref{ThmMain1}.

\begin{cor}\label{CorMain1} Let $m, D\in\N$, $D \neq 1$ squarefree, and $F=\Q(\!\sqrt{D})$. Let $L$ be a classical totally positive definite quadratic $\co_F$-lattice which represents every element of $m\co_F^{+}$. Then there exist primes $p_{1},p_{2}\leq 2m^{2}$, integers $2 \leq C_{i}\leq \gamma_{p_{i}}$, primes $q_{1},q_{2}$ such that $q_{i}\leq 2C_{i}m^{3}$ and integers $1 \leq D_{i}\leq q_{i}\gamma_{q_{i}}$ with the following properties:
\begin{enumerate}
 \item If $\rank L\leq 3$, then $\sqrt{\Delta_{D}} < 5C_{1}m^2$.
 \item If $\rank L\leq 4$, then $\sqrt{\Delta_{D}} < 9C_1m^2$.
 \item If $\rank L\leq 5$, then $\sqrt{\Delta_{D}} < 45 C_{2}m^3$.
 \item If $\rank L\leq 6$, then $\sqrt{\Delta_{D}} < \max\bigl\{45 C_{2}m^3, 5D_{1}\max\{C_{1},C_{2}\} m^2\bigr\}$.
 \item If $\rank L\leq 7$, then $\sqrt{\Delta_{D}} <\max\bigl\{5\max\{C_{1},D_{1}\}\max\{C_{2},D_{2}\}m^2,200C_{2}D_{2}m^4\bigr\}$.
\end{enumerate}
\end{cor}
\begin{proof}
Let us consider each case separately.

\smallskip

$\rank L \le 3$: Here, Theorem \ref{ThmMain1}(1) and Lemma \ref{le:ineqSlayer} immediately give $\sqrt{\Delta_{D}}<4C_{1}m^{2}+1 < 5C_{1}m^{2}$.

\medskip

$\rank L \le 4$: This part is analogous to the previous one. The only difference is that we use Lemma \ref{le:ineqSlayer} with $t=2$ instead of $t=1$.

\medskip

$\rank L \le 5$: By Theorem \ref{ThmMain1}(3), there are two possibilities in this case. The former gives
\begin{align*}
    \sqrt{\Delta_{D}}<4C_{1}C_{2}m^{2}+1 < 5C_{1}C_{2}m^{2}
\end{align*}
by Lemma \ref{le:ineqSlayer}. In the latter, since $C_2 \geq 2$ and $\Delta_D \geq 5$, we get
\begin{align*}
    \sqrt{\Delta_{D}}^{3}\leq 36C_{2}m^{3} \sqrt{\Delta_{D}}^{2} + 18C_2m^{3} \sqrt{\Delta_{D}} + m^{3} \leq\Bigl(36+\frac{18}{\sqrt5} +\frac{1}{10} \Bigr)C_{2}m^{3} \sqrt{\Delta_{D}}^{2} < 45 C_{2}m^{3} \sqrt{\Delta_{D}}^{2},
\end{align*}
so $\sqrt{\Delta_{D}}< 45 C_{2}m^{3}$. Therefore, we have
\begin{align*}
\sqrt{\Delta_{D}}< \max\{5C_{1}C_{2}m^2, 45 C_{2}m^3\} = C_2m^2\max\{5C_{1}, 45 m\}.
\end{align*}
However, we easily see that the second term is always bigger. If $p_1=2$, then $C_1\leq 7$ and the inequality holds; otherwise $C_1 \leq \gamma_{p_1} \leq 1.4 (2m^2)^{1/4} \log(2m^2)$ by Lemma \ref{LemNonresidueBound} below, and for positive $m$, this can be bounded by $3m$. Thus, $\max\{5C_1,45 m\} = 45 m$.

This leads to the conclusion that if $\rank L\leq 5$, then $\sqrt{\Delta_{D}}< 45 C_{2}m^3$.

\medskip

$\rank L \le 6$: The last inequality from Theorem \ref{ThmMain1}(4) can be treated using Lemma \ref{le:ineqSlayer} the same way as before. Together with the bounds from the case of $\rank L  \le 5$ we get that if $\rank L\leq 6$ then $\sqrt{\Delta_{D}} < \max\bigl\{45 C_{2}m^3, 5D_{1}\max\{C_{1},C_{2}\} m^2\bigr\}$.

\medskip

$\rank L \le 7$: From reasoning analogous to the previous parts, we have that if the first, second and third inequality from Theorem \ref{ThmMain1}(5) are satisfied, then respectively:
\begin{itemize}
    \item $\Delta_{D}\leq 4C_{1}D_{1}m^{4}$,
    \item $\sqrt{\Delta_{D}}< 5\max\{C_{1},D_{1}\}\max\{C_{2},D_{2}\}m^2$,
    \item $\sqrt{\Delta_{D}}< 45 C_{2}m^3$.
\end{itemize}
For the last one, note that since $C_2\geq 2$ and $\Delta_{D}\geq 5$, we have
\begin{align*}
    \sqrt{\Delta_{D}}^{4} &\leq 192C_{2}D_{2}m^{4}\sqrt{\Delta_{D}}^{3} + 144C_{2}D_{2}m^{4}\sqrt{\Delta_{D}}^{2} + 96\max\{C_{2},D_{2}\}m^{4}\sqrt{\Delta_{D}} + m^{4}\\
    &< 192C_{2}D_{2}m^{4}\sqrt{\Delta_{D}}^{3}+\Bigl(144 + \frac{96}{\sqrt{5}} + \frac{1}{10}\Bigr)C_{2}D_{2}m^{4}\sqrt{\Delta_{D}}^{2} \\
    &<192C_{2}D_{2}m^{4}\sqrt{\Delta_{D}}^{3} + 188C_{2}D_{2}m^{4}\sqrt{\Delta_{D}}^{2}.
\end{align*}
Hence, in this case,
\[
    \sqrt{\Delta_{D}}^{4} < 192C_{2}D_{2}m^{4}\sqrt{\Delta_{D}}^{3}+188C_{2}D_{2}m^{4}\sqrt{\Delta_{D}}^{2},
\]
from which $\sqrt{\Delta_D}<(96+\sqrt{96^{2}+188})C_{2}D_{2}m^{4}< 200C_{2}D_{2}m^4$.

Therefore, if $\rank L\leq 7$, then
\[
    \sqrt{\Delta_{D}} <\max\left\{5\max\{C_{1},D_{1}\}\max\{C_{2},D_{2}\}m^2,200C_{2}D_{2}m^4\right\},
\]
since $5\max\{C_{1},D_{1}\}\max\{C_{2},D_{2}\}m^2>\sqrt{4C_{1}D_{1}}m^2$ and $200C_{2}D_{2}m^4 > 45 C_{2}m^3$.
\end{proof}

Let us prove the cases $m=1$ and $m=2$ of Theorem \ref{MAINTHEOREM1} separately since they are easier than the general case.

\begin{proof}[Proof of Theorem $\ref{MAINTHEOREM1}$ for $m\leq 2$]
From the assumptions of Corollary \ref{CorMain1} (which are the same as for Theorem \ref{ThmMain1}), the following bounds for quantities $p_{i}$, $q_{i}$, $C_{i}$ and $D_{i}$ are clear if $m\in\{1,2\}$:
\begin{equation}\label{BoundsCpqm}
\begin{split}
    p_{i} &\leq \left\{\begin{array}{ll}
        2 & \text{if } m=1, \\
        7 & \text{if } m=2,
    \end{array}\right. \\
    C_{i} &\leq\left\{\begin{array}{ll}
        5 & \text{if } p_{i}=2 \qquad \text{(by Remark \ref{rem5not7}}), \\
        \gamma_{p_{i}}\leq 3 & \text{if } p_{i}\in\{3,5,7\},
    \end{array}\right. \\
    q_{i} &\leq 2C_{i}m^{3},\\
    D_{i} &\leq q_{i}\gamma_{q_{i}}.
\end{split}
\end{equation}

We begin with the case $m=1$. Inequalities \eqref{BoundsCpqm} imply $q_{i}\leq 10$, that is, either $q_{i}=2$ and $D_{i}\leq 14$, or $3\leq q_{i}\leq 7$ and $D_{i}\leq 21$. Therefore, Corollary \ref{CorMain1} directly implies that
\begin{itemize}
    \item if $\rank L\leq 3$, then $\sqrt{\Delta_{D}} < 5^2$,
    \item if $\rank L\leq 4$, then $\sqrt{\Delta_{D}} < 3^2\cdot 5$,
    \item if $\rank L\leq 5$, then $\sqrt{\Delta_{D}} < 3^2 \cdot 5^2$,
    \item if $\rank L\leq 6$, then $\sqrt{\Delta_{D}} < \max\left\{3^2 \cdot 5^2, 5^2\cdot 21\right\}=3\cdot 5^2\cdot 7$,
    \item if $\rank L\leq 7$, then $\sqrt{\Delta_{D}} < \max\left\{5\cdot 21\cdot 21, 200\cdot 5\cdot 21\right\}=2^3\cdot 3\cdot 5^3 \cdot 7$.
\end{itemize}

Let us assume that $m=2$. If $p_{i}=2$, then \eqref{BoundsCpqm} implies $q_{i}\leq 80$, and using the table \cite{Ben} we get $D_{i}\leq 7\cdot 71$. Similarly, if $p_{i}\in\{3,5,7\}$, then $q_{i}\leq 48$ and $D_{i}\leq 5\cdot 47$. Corollary \ref{CorMain1} again gives that
\begin{itemize}
    \item if $\rank L\leq 3$, then $\sqrt{\Delta_{D}}< 2^2\cdot 5^2$,
    \item if $\rank L\leq 4$, then $\sqrt{\Delta_{D}}< 2^2\cdot 3^2\cdot 5$,
    \item if $\rank L\leq 5$, then $\sqrt{\Delta_{D}}< 2^3\cdot 3^2 \cdot 5^2$,
    \item if $\rank L\leq 6$, then $\sqrt{\Delta_{D}}< \max\left\{2^3\cdot 3^2 \cdot 5^2, 2^2\cdot 5^2\cdot 7\cdot 71\right\}=2^2\cdot 5^2\cdot 7\cdot 71$,
    \item if $\rank L\leq 7$, then $\sqrt{\Delta_{D}}< \max\left\{2^2\cdot 5\cdot 7^2\cdot 71^2, 2^4\cdot 200\cdot 5\cdot 7\cdot 71\right\}=2^7\cdot  5^3\cdot 7\cdot 71$.\qedhere
\end{itemize}
\end{proof}

Before we explain how the case of $m\geq 3$ of Theorem \ref{MAINTHEOREM1} follows from Theorem \ref{ThmMain1} and Corollary \ref{CorMain1}, we will need the following bounds for the least quadratic non-residue.

\begin{lem}\label{LemNonresidueBound}
Let $p$ be an odd prime. Then for every $\varepsilon >0$:
\begin{enumerate}
\item unconditionally: $\gamma_{p}\ll_{\varepsilon} p^{\frac{1}{4\sqrt{\mathrm{e}}}+\varepsilon}$ (Burgess's bound), \label{it:Burgess}
\item under GRH: $\gamma_{p}< 2(\log p)^{2}$. \label{it:GRHexpli}
\item unconditionally: $\gamma_p < 1.4 p^{1/4} \log p$. \label{it:uncond}
\end{enumerate}
\end{lem}
\begin{proof}
The first bound is from \cite[Theorem 2]{Bu}. The bounds (2) and (3) are trivial for $p=3$ and for $p\geq 5$ follow from \cite[Corollary 1.1]{LLS} and \cite{Tre}, respectively.
\end{proof}

\begin{proof}[Proof of Theorem $\ref{MAINTHEOREM1}$ for $m\geq 3$]
We will need bounds in terms of $m$ for quantities $p_{i}$, $q_{i}$, $C_{i}$ and $D_{i}$, valid for $m \geq 3$. If $p_{i},q_{i}\geq 3$ we have by Lemma \ref{LemNonresidueBound}\eqref{it:uncond}:
\begin{equation}\label{BoundsCpq}
\begin{split}
p_{i} & \leq 2m^{2}, \\
C_{i} & \leq \gamma_{p_{i}} < 1.4p_{i}^{1/4}\log p_{i} < 5m^{1/2}\log m, \\
q_{i} & \leq 2C_{i}m^{3} < 10m^{7/2}\log m, \\
\gamma_{q_{i}} & < 1.4q_{i}^{1/4}\log q_{i} <1.4\cdot 10^{1/4} m^{7/8}(\log m)^{1/4} \log\left( 10m^{7/2}\log m \right)\\
& <1.4\cdot 10^{1/4} m^{7/8}(\log m)^{1/4} \log\left( m^{\log 10/\log 3 + 7/2 + 1}\right) < 20 m^{7/8}(\log m)^{5/4}, \\
D_{i} & \leq q_{i}\gamma_{q_{i}} < 200m^{35/8}(\log m)^{9/4}.
\end{split}
\end{equation}
In the case of $p_{i}=2$, we still have $C_{i}\leq \gamma_2 = 7 < 5m^{1/2}\log m$, as $m\geq 3$. Similarly, if $q_{i}=2$, then $D_{i}\leq 2\cdot \gamma_2 = 2\cdot7 < 200m^{35/8}(\log m)^{9/4}$. Therefore we can use inequalities \eqref{BoundsCpq} for all $p_{i}$ and $q_{i}$.

Now denote $\kappa:=\frac{1}{\sqrt{\mathrm{e}}}\approx 0.6065$ and $\lambda:=\frac{\kappa^{2}}{4}+\frac{5\kappa}{2}-1\approx 0.6083$. Inequalities \eqref{BoundsCpq} and Lemma \ref{LemNonresidueBound}\eqref{it:Burgess} give inexplicit bounds:
\begin{equation}\label{BoundCpqInex}
\begin{split}
    C_{i} &\leq \gamma_{p_{i}} \ll_{\varepsilon} p_{i}^{\frac{1}{4\sqrt\mathrm{e}}+\frac{\varepsilon}{2}} \ll m^{\frac{1}{2\sqrt\mathrm{e}}+\varepsilon} = m^{\frac{\kappa}{2}+\varepsilon}, \\
    D_i &\leq q_i\gamma_{q_i} \ll_{\varepsilon} q_i \cdot q_i^{\kappa/4+\varepsilon} \leq (2C_im^3)^{1+\kappa/4+\varepsilon} \ll_{\varepsilon} (m^{\kappa/2+\varepsilon} \cdot m^3)^{1+\kappa/4+\varepsilon} = m^{(3+\kappa/2+\varepsilon)(1+\kappa/4+\varepsilon)}\\ &= m^{3+5\kappa/4 + \kappa^2/8 + \tilde{\varepsilon}} = m^{7/2+\lambda/2+\tilde{\varepsilon}}.
\end{split}
\end{equation}

Similarly, for the conditional bounds we use Lemma \ref{LemNonresidueBound}\eqref{it:GRHexpli}:
\begin{equation}\label{BoundCpqCond}
\begin{split}
    C_{i} & \leq \gamma_{p_{i}} < 2(\log p_{i})^{2}\leq 2\bigl(\log (2m^{2})\bigr)^{2}\leq 2\cdot 3^{2} (\log m)^{2} \ll (\log m)^{2}, \\
    D_{i} &\leq 2C_i m^3 \cdot \gamma_{q_{i}} \ll (\log m)^{2}m^3 \cdot (\log q_{i})^{2} \ll m^{3} (\log m)^{4};
\end{split}
\end{equation} 
again, the bound $C_i \leq 2\cdot 3^2 (\log m)^2$ holds even if $p_i=2$. 

We are ready to find the bounds on $\Delta_D$. The cases $\rank L\leq 3$, $\rank L\leq 4$ and $\rank L\leq 5$ follow immediately from parts (1), (2) and (3) of Corollary \ref{CorMain1} by using the bounds for $C_1, C_2$ from \eqref{BoundsCpq}, \eqref{BoundCpqInex} and \eqref{BoundCpqCond}.

Now suppose that $\rank L\leq 6$. Hypothetically it may happen that the first term in the expression $\max\bigl\{45 C_2m^3, 5D_1\max\{C_1,C_2\}m^2\bigr\}$ is the bigger one. However, if we replace $C_1, C_2$ and $D_1$ by our bounds from \eqref{BoundsCpq}, \eqref{BoundCpqInex} and \eqref{BoundCpqCond}, then clearly the second term will be the important one, since it contains the bound for $D_1$, which is large. Therefore, to shorten the text, we use only the second term here. Explicit bounds \eqref{BoundsCpq} give
\begin{align*}
    \sqrt{\Delta_D} &< \max\bigl\{45 C_2m^3,5D_{1}\max\{C_{1},C_{2}\} m^2\bigr\} \\
    & < 5\cdot 200 m^{35/8}(\log m)^{9/4} \cdot 5 m^{1/2}\log m \cdot m^2 = 2^3\cdot 5^4 m^{55/8}(\log m)^{13/4}.
\end{align*}
For unconditional but inexplicit bound, we use \eqref{BoundCpqInex}. For every $\varepsilon>0$:
\[
    \sqrt{\Delta_D} \ll \max\bigl\{C_2m^3,D_{1}\max\{C_{1},C_{2}\} m^2\bigr\} \ll_{\varepsilon} m^{7/2+\lambda/2+\varepsilon} \cdot m^{\kappa/2 + \varepsilon} \cdot m^2 = m^{(11+\kappa+\lambda)/2+2\varepsilon}.
\]

Analogously, under the GRH we have by \eqref{BoundCpqCond}:
\[
    \sqrt{\Delta_{D}}\ll m^3(\log m)^4\cdot (\log m)^2\cdot m^2 = m^5(\log m)^6.
\]

If $\rank L\leq 7$, then using the inequalities \eqref{BoundsCpq}, Corollary \ref{CorMain1}(5) yields
\begin{align*}
\sqrt{\Delta_{D}} &< \max\bigl\{5\max\{C_{1},D_{1}\}\max\{C_{2},D_{2}\}m^2,200C_{2}D_{2}m^4\bigr\}\\
&< \max\bigl\{ 5 \cdot \bigl( 200m^{35/8}(\log m)^{9/4} \bigr)^2 \cdot m^2 , 200 \cdot 5m^{1/2}\log m \cdot 200m^{35/8}(\log m)^{9/4} \cdot m^4 \bigr\}\\
&= 5\cdot 200^{2}m^{35/8}(\log m)^{9/4}\cdot m^2\cdot  \max\bigl\{m^{35/8}(\log m)^{9/4} , m^{1/2}\log m \cdot m^2 \bigr\}\\
&= 5\cdot200^2 m^{35/8}(\log m)^{9/4} \cdot m^2 \cdot m^{35/8}(\log m)^{9/4}\\
&= 5\cdot200^2 m^{43/4}(\log m)^{9/2}.
\end{align*}

The inexplicit bounds follow from Corollary \ref{CorMain1}(5) and inequalities \eqref{BoundCpqInex} and \eqref{BoundCpqCond}:
\begin{align*}
    \sqrt{\Delta_{D}} &\ll_{\varepsilon} \max\bigl\{m^{7/2+\lambda/2 +\varepsilon} \cdot m^{7/2+\lambda/2 +\varepsilon} \cdot m^2, m^{\kappa/2+\varepsilon} \cdot m^{7/2+\lambda/2 +\varepsilon} \cdot m^4\bigr\} = m^{9+\lambda+2\varepsilon}, \\
    \sqrt{\Delta_{D}} &\ll \max\left\{m^3(\log m)^4\cdot m^3(\log m)^4\cdot m^2, (\log m)^2\cdot m^3(\log m)^4\cdot m^4 \right\} = m^8(\log m)^8.
\end{align*}
This concludes the proof in the case of $m\geq 3$.
\end{proof}

In this highly technical section, we showed how to deduce the main Theorem \ref{MAINTHEOREM1} from the more subtle version, namely Theorem \ref{ThmMain1}. The proof of Theorem \ref{ThmMain1} itself is divided into five parts, namely Theorems \ref{rank 3}, \ref{rank 4}, \ref{rank 5}, \ref{rank 6} and \ref{rank 7}. Before we start with its proof, we need to prepare several auxiliary results.

\section{Technical estimates for linear independence}\label{sec:tech}

In this section, we start the proof of Theorem \ref{ThmMain1} by presenting a series of lemmas necessary for the proofs in Sections \ref{sec:45} and \ref{sec:678}. 

The following simple but useful observation is specific to working over real quadratic fields. Also, note that part (2) requires the lattice to be classical, and that is the reason why we need the same assumption in our main Theorem \ref{ThmMain1}.

\begin{lem}\label{LemNonIntBound}
Let $D>1$ be a squarefree integer, and $F=\Q(\!\sqrt{D})$.

 \begin{enumerate}
     \item If $\alpha\in \co_F\setminus \Z$, then $\max\{|\alpha|,|\alpha'|\}\geq \frac{\sqrt{\Delta_{D}}}{2}$.
     \item Let $L$ be a classical totally positive definite $\co_F$-lattice and $v,w \in L$ such that $Q(v)Q(w) \prec \frac{\Delta_D}4$. Then $\B(v,w) \in \Z$.
 \end{enumerate}
\end{lem}
\begin{proof}
(1) First assume $D \equiv 1 \pmod4$. We can write $\alpha = \frac{a+b\sqrt{D}}{2}$ with $a,b \in \Z$. Since $\alpha \notin\Q$, we have $b \neq 0$. Thus
\[
\max\{|\alpha|,|\alpha'|\} = \frac{|a|+|b|\sqrt{D}}{2} \geq \frac{0+\sqrt{D}}{2} = \frac{\sqrt{\Delta_{D}}}{2},
\]
just as we wanted. If $D \equiv 2,3 \pmod4$, then analogously, $\max\{|\alpha|,|\alpha'|\} \geq \sqrt{D} = \frac{\sqrt{\Delta_D}}{2}$.

(2) The Cauchy--Schwarz inequality yields $\B(v,w)^2 \preceq Q(v)Q(w) \prec \frac{\Delta_D}4$. This precisely says that $\max \{|\B(v,w)|, |\B(v,w)'|\} < \frac{\sqrt{\Delta_D}}{2}$. Thus $\B(v,w) \in \Z$ by (1).
\end{proof}

\smallskip

For the rest of the section, let $F$ be any totally real number field. Remember our convention: Whenever, in a totally positive definite quadratic lattice, vectors $v_1, \ldots, v_k$ are given, we denote the Gram matrix $\begin{pmatrix} \B(v_i,v_j) \end{pmatrix}_{1 \le i, j \le k}$ as $\Gr{k}$. This matrix is always positive semidefinite. Now, let us recall the following well-known property.

\begin{lem} \label{gram}
Let $L$ be a totally positive definite $\co_F$-lattice. Then the following conditions for vectors $v_1,\ldots,v_n \in L$ are equivalent:
\begin{enumerate}
    \item $v_1,\ldots,v_n$ are linearly independent,
    \item the Gram matrix $\Gr{n}$ is positive definite,
    \item $\det \Gr{n} \neq 0$.
\end{enumerate}
The same equivalences hold if the quadratic lattice $L$ is replaced by a quadratic space $V$.
\end{lem}

Note that the Gram matrix is in fact \emph{totally} positive semidefinite, i.e., positive semidefinite in all embeddings; but according to the lemma, to test linear independence, it is enough to check the positive definiteness in one embedding only. This is quite clear, but worth bearing in mind.

Before continuing, observe that if $\Gr{k} \in M_k(\Q)$, then the $\Q$-span of the vectors $v_i$ is a well-defined quadratic space over $\Q$. The next auxiliary result studies this situation: it is essentially an argument which deduces linear independence over $F$ from independence over $\Q$.

\begin{lem} \label{indOverF}
Let $\alpha \in \Z$. Let $L$ be a totally positive definite $\co_F$-lattice and $v_1, \ldots, v_k \in L$ be linearly independent vectors such that $\Gr{k} \in M_{k}(\Q)$ and that the quadratic space $\Q v_1 + \cdots + \Q v_k$ does not represent $\alpha$. Assume further that $v_{k+1} \in L$ is such that $Q(v_{k+1})=\alpha$ and $\Gr{k+1} \in M_{k+1}(\Q)$. Then $v_1, \ldots, v_{k+1}$ are linearly independent (over $F$).    
\end{lem}
\begin{proof}
Assume for the sake of contradiction that $v_1, \ldots, v_{k+1}$ are linearly dependent. Then by Lemma \ref{gram}, $\det \Gr{k+1} =0$. Since it is a rational matrix, there is a vector $0 \neq a = (\alpha_1, \ldots, \alpha_{k+1})^{\mathrm{T}} \in \Z^{k+1}$ such that $a^{\mathrm{T}}\Gr{k+1}a = 0$. Thus we have
\[
0 = a^{\mathrm{T}}\Gr{k+1}a = Q(\alpha_1 v_1 + \cdots + \alpha_{k+1} v_{k+1}).
\]
Since $L$ is positive definite, this gives $\alpha_1 v_1 + \cdots + \alpha_{k+1} v_{k+1} = 0$. Further, linear independence of $v_1, \ldots, v_k$ implies $\alpha_{k+1} \neq 0$. Therefore, 
\[
Q\bigl((\alpha_1 v_1 + \cdots + \alpha_k v_k)/\alpha_{k+1}\bigr) = Q(-v_{k+1}) = \alpha,
\]
which is a contradiction, as $\Q v_1 + \cdots + \Q v_k$ does not represent $\alpha$.
\end{proof}

For the next crucial lemma, remember that $\gamma_p$ denotes the smallest non-residue modulo an odd prime $p$. 

\begin{lem} \label{Clem}
Let $(V,Q)$ be a positive definite quadratic space over $\Q$ and $x_1, x_2 \in V$ be linearly independent vectors such that the Gram matrix $G = \bigl(\B(x_i,x_j)\bigr)_{1\leq i,j \leq 2} \in M_2(\Z)$. Then there exists an integer $C \in \Z^+$ such that $C\cdot Q(x_1)$ is not represented by $\Q x_1 + \Q x_2$. Moreover, we can choose
\begin{enumerate}
\item $C=3$ if $\det G$ is a square,
\item $C=\gamma_{p}$ if $p$ is an odd prime such that ${\ord}_{p}\bigl(\det G \bigr)$ is odd,
\item $C=5$ if ${\ord}_{2}\bigl(\det G \bigr)$ is odd.
\end{enumerate}
\end{lem}
\begin{proof}
Observe that always at least one of the three listed cases occurs. We shall consider them separately. By Gram--Schmidt process, we may diagonalize the quadratic space generated by $x_1,x_2$: $\Q x_1+\Q x_2 =\Q x_1 \perp \Q y$ with $Q(x_1),Q(y) \in \Z^+$; then $Q(x_1)Q(y) \in \det G \cdot (\Q^{\times})^2$.

\smallskip

(1) Since $\det G$ is a square and $Q(x_1)Q(y) \in \det G \cdot (\Q^{\times})^2$, we have $\ord_3(Q(x_1)) \equiv \ord_3(Q(y)) \pmod 2$. Moreover, for every $\alpha\in\Z\setminus 3\Z$ we have $\alpha^{2}\equiv 1\pmod{3}$, so, in order for $\det G$ to be a square modulo $3$, it must hold that \[\frac{Q(x_1)}{3^{\ord_3(Q(x_1))}} \cdot \frac{Q(y)}{3^{\ord_3(Q(y))}}\equiv 1\pmod{3};\] this trivially yields
\[
\frac{Q(x_1)}{3^{\ord_3(Q(x_1))}} \equiv \frac{Q(y)}{3^{\ord_3(Q(y))}} \pmod{3}.\]
Therefore, for every $x=\alpha_1 x_1+\alpha_2 y \in \Q x_1 \perp \Q y=\Q x_1+\Q x_2$ we have 
\[
\ord_3(Q(x)) =\ord_3(Q(\alpha_1 x_1+\alpha_2 y))=\ord_3(\alpha_1^2Q(x_1)+\alpha_2^2Q(y)) =\min\{\ord_3(\alpha_1^2Q(x_1)), \ord_3(\alpha_2^2Q(y))\};
\]
the last equality is based on the fact that even if the summands have the same $3$-valuation $\nu$, then modulo $3^{\nu+1}$ we get $3^{\nu}(1+1)$, which is not divisible by $3^{\nu+1}$.
Further, we have
\begin{align*}
\ord_3(\alpha_1^2Q(x_1))=2\ord_3(\alpha_{1}) + \ord_{3}(Q(x_1))\equiv \ord_{3}(Q(x_1)) \pmod{2}
\end{align*}
and similarly $\ord_3(\alpha_2^2Q(y)) \equiv \ord_{3}(Q(y)) \equiv \ord_{3}(Q(x_1)) \pmod{2}$. Putting everything together, we get
\begin{align*}
\ord_3(Q(x))\equiv \ord_{3}(Q(x_1)) \pmod{2}.
\end{align*}

Thus we have shown that $\Q x_1 + \Q x_2$ only represents elements which have the same parity of $3$-valuation as $Q(x_1)$. Clearly $\ord_3\bigl(3 Q(x_1)\bigr) = 1 + \ord_3\bigl(Q(x_1)\bigr)$, so $\Q x_1 + \Q x_2$ fails to represent $3Q(x_1)$.

\smallskip

(2) Since $\ord_p\bigl(\det G \bigr) \equiv 1 \pmod 2$, we have $\ord_p(Q(x_1)) \not\equiv \ord_p(Q(y)) \pmod 2$.
Let us consider any $x=\alpha_1x_1+\alpha_2y \in \Q x_1 \perp \Q y=\Q x_1+\Q x_2 $ such that
\[
\ord_p(Q(x)) \equiv \ord_p(Q(x_1)) \pmod 2.
\]

Observe that
\[
\ord_p(Q(x))=\ord_p(Q(\alpha_1x_1+\alpha_2y))=\ord_p(\alpha_1^2Q(x_1)+\alpha_2^2Q(y))=\min\{\ord_p(\alpha_1^2Q(x_1)),\ord_p(\alpha_2^2Q(y))\},    
\]
since we have $\ord_p(\alpha_1^2 Q(x_1)) \neq \ord_p(\alpha_2^2 Q(y))$ due to different parity. Hence it must hold that
\[
\min\{\ord_p(\alpha_1^2Q(x_1)), \ord_p(\alpha_2^2Q(y))\}=\ord_p(\alpha_1^2Q(x_1)),
\]
as otherwise we would get $\ord_{p}(Q(x)) \equiv \ord_{p}(Q(y))\not\equiv \ord_p(Q(x_1)) \pmod{2}$, which is a contradiction.

The above consideration implies $\ord_p(\alpha_1^2Q(x_1)) < \ord_p(\alpha_2^2Q(y))$, and so $\frac{\alpha_2^2Q(y)}{\alpha_1^2Q(x_1)} \in p \Z_p$. Therefore
\[
Q(x)=Q(\alpha_1x_1+\alpha_2y)=\alpha_1^2Q(x_1)+\alpha_2^2Q(y) = \alpha_1^2Q(x_1) \Bigl(1+\frac{\alpha_2^2Q(y)}{\alpha_1^2Q(x_1)}\Bigr) \in \alpha_1^2Q(x_1) \cdot \left(1+p\Z_p\right).    
\]

By Lemma \ref{local value lem}, we have $1+p\Z_p = (1+p\Z_p)^2 \subset (\Z_p^{\times})^2$. Therefore, $Q(x)\in \alpha_1^2 Q(x_{1}) \cdot (\Z_p^{\times})^2$.

Hence we have shown that for every $x \in \Q x_1 + \Q x_2$, if $\ord_p(Q(x)) \equiv \ord_p(Q(x_1)) \pmod 2$, then $Q(x)\in Q(x_{1}) \cdot (\Q_p^{\times})^2$. Now we see that $Q(x) \neq \gamma_p Q(x_1)$: On the one hand, we have just proved it under the assumption $\ord_p(Q(x)) \equiv \ord_p(Q(x_1)) \pmod 2$; on the other hand, since $\ord_p(\gamma_p)=0$, it cannot hold if $\ord_p(Q(x)) \not\equiv \ord_p(Q(x_1)) \pmod 2$ either. Thus the quadratic space indeed fails to represent $\gamma_p Q(x_1)$.

\smallskip

(3) We proceed similarly as in the previous case. Since $\ord_2\bigl(\det G \bigr) \equiv 1 \pmod 2$, we have $\ord_2(Q(x_1)) \not\equiv \ord_2(Q(y)) \pmod 2$.
Consider any $x=\alpha_1x_1+\alpha_2y \in \Q x_1 \perp \Q y=\Q x_1+\Q x_2 $ such that
\[
\ord_2(Q(x)) \equiv \ord_2(Q(x_1)) \pmod 2.
\] 
We have $\ord_2(\alpha_1^2Q(x_1)) < \ord_2(\alpha_2^2Q(y))$ by the same reasoning as in the previous case. This, together with $\ord_2(Q(x_1)) \not\equiv \ord_2(Q(y)) \pmod 2$, implies that $\frac{\alpha_2^2Q(y)}{\alpha_1^2Q(x_1)} \in 2^k \Z_2^{\times}$ for some odd $k\in\N$. 
Since $Q(x)=\alpha_1^2Q(x_1)+\alpha_2^2Q(y)=\alpha_1^2Q(x_1) \cdot \Bigl(1+\frac{\alpha_2^2Q(y)}{\alpha_1^2Q(x_1)}\Bigr)$ with $\frac{\alpha_2^2Q(y)}{\alpha_1^2Q(x_1)} \not\equiv 4 \pmod 8$, we can deduce 
\begin{align}\label{NotEquivMod8}
\frac{Q(x)}{2^{\ord_2(Q(x))}} \not\equiv \frac{Q(x_1)}{2^{\ord_2(Q(x_1))}} \cdot 5 \pmod 8.
\end{align}

If we now assume that $Q(x) = 5 Q(x_1)$, we get
\[
\frac{Q(x)}{2^{\ord_2(Q(x))}} = \frac{5Q(x_1)}{2^{\ord_2(5Q(x_1))}} = 5 \frac{Q(x_1)}{2^{\ord_2(Q(x_1))}},
\]
which directly contradicts \eqref{NotEquivMod8}. Thus $5 Q(x_1)$ is not represented, as we needed.
\end{proof}

Note that although in several of the upcoming lemmas we do not require $L$ to be classical, we always work with a sublattice generated by vectors whose Gram matrix has elements in $\Z$, which implies that this sublattice is classical.

\begin{lem} \label{C1lem}
Let $L$ be a totally positive definite $\co_F$-lattice and $v_{1},v_{2}\in L$ linearly independent vectors such that the Gram matrix $\Gr{2} \in M_2(\Z)$. Then there exists $C_1 \in \N$ such that for any $v_3 \in L$ that satisfies $Q(v_3)=C_1 \cdot Q(v_1)$ and $\Gr3 \in M_3(\Z)$, the vectors $v_1,v_2,v_3$ are linearly independent. Moreover, we can take:
\begin{enumerate}
\item $C_{1}=3$ if $\det \Gr{2}$ is a square,
\item $C_{1}=\gamma_{p}$ if $p$ is an odd prime such that ${\ord}_{p}\bigl(\det \Gr2 \bigr)$ is odd,
\item $C_{1}=5$ if ${\ord}_{2}\bigl(\det \Gr2 \bigr)$ is odd.
\end{enumerate}
\end{lem}

\begin{proof}
Since all $\B(v_i,v_j)$ are integers, $\Q v_1 + \Q v_2$ is a well-defined quadratic space over $\Q$. Lemma \ref{Clem} shows the existence of a $C$ such that this quadratic space fails to represent $C\cdot Q(v_1)$, and specifies that we can choose this $C$ exactly as in the present lemma. Let us put $C_1 = C$ and take any vector $v_3 \in L$ such that $Q(v_3) = C_1\cdot Q(v_1)$ and $\Gr3 \in M_3(\Z)$. Now it only remains to prove that the vectors $v_1, v_2, v_3$ are indeed linearly independent. But this is precisely the content of Lemma \ref{indOverF} for $k=2$.
\end{proof}

We will also need another, even more direct corollary of Lemma \ref{Clem}.

\begin{cor} \label{C2cor}
Let $M=\bigl(\begin{smallmatrix} m_{11} & m_{12} \\ m_{12} & m_{22} \end{smallmatrix}\bigr) \in M_2(\Z)$ be positive definite and different from $\bigl(\begin{smallmatrix} 1 & 0 \\ 0 & 1 \end{smallmatrix}\bigr)$. Then there exists a prime $p_{2} \leq m_{11}m_{22}$ and an integer $2 \leq C_2 \leq \gamma_{p_{2}}$ such that any positive semidefinite matrix of the block form
\[
\begin{pmatrix}  M & v \\ v^{\mathrm{T}} & C_2m_{11} \end{pmatrix} \in M_3(\Z)
\]
is in fact positive definite.
\end{cor}
\begin{proof}
If we rewrite Lemma \ref{Clem} in terms of matrices instead of quadratic spaces (e.g., using Lemma \ref{gram}), we get almost precisely the present statement, only instead of the inequalities for $p_2$ and $C_2$ we get the description contained in Lemma \ref{Clem}: Either $m_{11}m_{22}-m_{12}^2 = 2^t$, and then $p_2=2$ and $C_2 = 3$ or $5$, or $p_2 \mid m_{11}m_{22}-m_{12}^2$ is odd and $C_2=\gamma_{p_2}$. In the latter case, the inequalities are clear; in the former, they are violated only if $m_{11}m_{22}=1$, which occurs only for the identity matrix.
\end{proof}

\begin{remark} \label{rem5not7}
From the proof it is clear that Corollary \ref{C2cor} holds even if we put $\gamma_2 = 5$ instead of $\gamma_2 = 7$. In fact, throughout the paper, all bounds of the form $C_i \leq \gamma_{p_i}$ hold in this setting, since they are all based on Lemma \ref{Clem}, either via Lemma \ref{C1lem} or via Corollary \ref{C2cor}. (This is in contrast with the bounds $D_i \leq q_i\gamma_{q_i}$, which need $\gamma_2=7$.) In particular, this is true for Corollary \ref{ind cor}(1) and thus for all results based on it.

However, since all this only plays a role in the proof of the explicit bounds for $m=1, 2$ in Theorem \ref{ThmMain1}, we have decided to keep the formulations simpler by putting $\gamma_2 = 7$ for the whole paper.
\end{remark}

Note that although we denoted the primes in Lemma \ref{Clem}, Lemma \ref{C1lem} and Corollary \ref{C2cor} differently, they are essentially the same prime and we have done so only for future reference. We now start another triple of statements, again two lemmas followed by a corollary, which have precisely the same mutual relations as the past three statements. This time we denote the primes by $q, q_1, q_2$.

\begin{lem} \label{Dlem}
Let $(V,Q)$ be a positive definite quadratic space over $\Q$ and $x_1, x_2, x_3 \in V$ be linearly independent vectors such that the Gram matrix $G = (\B(x_i,x_j))_{1\leq i,j \leq 3} \in M_3(\Z)$. Then there exists a prime $q \mid 2 \det G$ and numbers $d \in \{0,1\}$ and 
\[
\eta \in
\begin{cases}
\text{ $\{1, \gamma_q\}$} & \text{if $q$ is odd},  \\
\text{ $\{1,3,5,7\}$} & \text{if $q=2$}, \\
\end{cases}
\]
such that $\Q x_1 + \Q x_2 + \Q x_3$ fails to represent all the elements of  $q^{d} \eta \cdot (\Q_q^{\times})^2$.
\end{lem}
\begin{proof}
By Hilbert's reciprocity law for the Hasse symbol, a positive definite ternary quadratic space over $\Q$ is anisotropic over $\Q_q$ for at least one (finite) prime $q$  \cite[58:6]{OM}. Thus, let $q$ be a prime number for which $\Q x_1 + \Q x_2 + \Q x_3$ is anisotropic over $\Q_q$. Note that $q=2$ or $q \mid \det G$. Indeed, if $p$ is an odd prime not dividing the determinant, then $\Z_px_1 + \Z_px_2 + \Z_px_3$ is a so-called \emph{unimodular} $\Z_p$-lattice \cite[82:13]{OM}, so it can be written as $\Z_p y_1 \perp \Z_p y_2 \perp \Z_p y_3$ with $Q(y_i) \in \Z_p^{\times}$ \cite[92:1]{OM}; thus it is isotropic over $\Q_p$ \cite[63:14]{OM}.

Thus $q \mid 2\det G$, and the anisotropic quadratic $\Q_q$-space $\Q_qx_1+\Q_qx_2+\Q_qx_3$ does not represent all of $\Q_q$ by Lemma \ref{aniso ternary}; hence it fails to represent at least one square class. If $q$ is odd, then by Corollary \ref{local value}(4), the four square classes are given as $q^{d} \eta \cdot (\Q_q^{\times})^2$ with $d\in \{0,1\}$ and $\eta \in \{1,\gamma_q\}$.
When $q=2$, then by Corollary \ref{local value}(3) there are eight square classes and we can find $d \in \{0,1\}$ and $\eta \in \{1,3,5,7 \}$ such that the elements of the square class $2^{d} \eta \cdot (\Q_2^{\times})^2$ are not represented.

In either case, we see that, in particular, the smaller quadratic space $\Q x_1 + \Q x_2 + \Q x_3$ fails to represent elements of $q^{d} \eta \cdot (\Q_q^{\times})^2$.
\end{proof}

\begin{lem} \label{almostD1lem}
Let $L$ be a totally positive definite $\co_F$-lattice and let $v_1,v_2,v_3 \in L$ be linearly independent vectors such that their Gram matrix $\Gr3 \in M_3(\Z)$. Then there exists a rational prime $q_1 \mid 2\det \Gr3$ and numbers $d \in \{0,1\}$ and 
\[
\eta \in
\begin{cases}
\text{ $\{1, \gamma_{q_1}\}$} & \text{if $q_1$ is odd},  \\
\text{ $\{1,3,5,7\}$} & \text{if $q_1=2$}, \\
\end{cases}
\]
with the following property: For every $v_4 \in L$ such that the Gram matrix $\Gr4 \in M_4(\Z)$ and
\[
Q(v_4) \in q_1^{d} \eta \cdot (\Q_{q_1}^{\times})^2,
\]
vectors $v_1,v_2,v_3,v_4$ are linearly independent.
\end{lem}
\begin{proof}
The proof is the same as for Lemma \ref{C1lem}: The conditions ensure that $\Q v_1 + \Q v_2 + \Q v_3$ is a well-defined quadratic space satisfying the assumptions of Lemma \ref{Dlem}; thus, it fails to represent $q_1^{d} \eta \cdot (\Q_{q_1}^{\times})^2$. From Lemma \ref{indOverF} follows the linear independence of $v_1, \ldots, v_4$. 
\end{proof}

\begin{cor} \label{D2cor}
Let $M=(m_{ij})_{1\le i,j \le 3}\in M_3(\Z)$ be a positive definite matrix different from the identity matrix. Then there exists a prime $q_{2} \leq m_{11}m_{22}m_{33}$ and an integer $1 \leq D_2 \leq q_2 \gamma_{q_2}$ such that any positive semidefinite matrix of the block form
\[
\begin{pmatrix}  M & v \\ v^{\mathrm{T}} & D_2m_{11} \end{pmatrix} \in M_4(\Z)
\]
is in fact positive definite.
\end{cor}

\begin{proof}
The proof is similar to that of Corollary \ref{C2cor}. Mainly, we need to translate Lemma \ref{Dlem} into matrix form. This yields almost exactly what we need, including the inequality $D_2 \leq q_2\gamma_{q_2}$; it only remains to show the inequality for $q_2$. We get that either $q_2=2$, or $q_2 \mid \det M$. In the latter case, we have $q_2 \leq \det M \leq m_{11}m_{22}m_{33}$ by Hadamard's inequality, see Lemma \ref{Hadamard}. In the former, the desired inequality only fails if $m_{11}m_{22}m_{33}=1$. However, this only happens for the identity matrix.
\end{proof}

Now we can apply the obtained lemmas to a specific choice of vectors $v_{1}$ and $v_{2}$.

\begin{cor}\label{ind cor}
Let $L$ be a totally positive definite quadratic $\co_F$-lattice and $m\in\N$. Let $v_{1}, v_{2} \in L$ be linearly independent vectors such that $Q(v_{1})=m$, $Q(v_{2})=2m$ and $\B(v_{1},v_{2})\in\Z$. Then:
\begin{enumerate}
\item There exists a prime number $p\leq 2m^{2}$ and an integer $2 \leq C_{1}\leq \gamma_{p}$ with the following property: If $v_{3}\in L$ is a vector such that $Q(v_{3})=C_{1}m$ and that $\Gr3 \in M_3(\Z)$, then vectors $v_{1},v_{2},v_{3}$ are linearly independent.
\item Further, there exists a prime number $q\leq 2C_{1}m^{3}$ and an integer $1 \leq D_{1}\leq q\gamma_{q}$ (both may depend on $v_3$) with the following property: If $v_{4}\in L$ is a vector such that $Q(v_{4})=D_{1}m$ and that $\Gr4 \in M_4(\Z)$, then vectors $v_{1},v_{2},v_{3},v_{4}$ are linearly independent.
\end{enumerate}
\end{cor}
\begin{proof}
The first part of the statement follows from Lemma \ref{C1lem}: If $\det \Gr2$ is a square or twice a square, we can take $p=2$. Otherwise, let $p$ be such that $\ord_{p}\bigl(\det \Gr2 \bigr)$ is odd. The existence of $C_{1}$ and the bound $2 \leq C_{1}\leq \gamma_{p}$ follow directly from the statement of Lemma \ref{C1lem}. We only need to prove the bound $p\leq 2m^{2}$. This is obvious if $p=2$. If $p$ is odd, then it is a divisor of $\det \Gr2$, so, indeed,
\begin{align*}
p\leq \det \Gr2 =Q(v_{1})Q(v_{2})-\B(v_{1},v_{2})^{2}\leq Q(v_{1})Q(v_{2})=2m^{2}.
\end{align*}

In order to prove the second part of the statement, we will use Lemma \ref{almostD1lem}. We need to choose $D_{1} \in \N$ so that $Q(v_{4})=D_{1}Q(v_{1})=D_{1}m\in q^{d} \eta \cdot (\Q_q^{\times})^2$, where $q$ is the prime from Lemma \ref{almostD1lem} and $d, \eta$ are the same as in that lemma. We could choose $D_{1}=q^{d}\eta \cdot m$. But we can do better: in the same square class as $q^{d}\eta \cdot m$, we will choose another representative $D_1 = q^{\tilde{d}} \tilde{\eta}$ with $\tilde{d} \in \{0,1\}$ and $\tilde{\eta} \in \{1,\gamma_q\}$ (or $\tilde{\eta} \in \{1,3,5,7\}$ for $q=2$). Then, obviously, $1 \leq D_1 \leq q\gamma_q$. In order to finish the proof, it is enough to show that $q\leq 2C_{1}m^{3}$. This is clear for $q=2$ and if $q\mid  \det \Gr3$, then $q\leq \det \Gr3 \leq m\cdot 2m\cdot C_{1}m$ by Hadamard's inequality (Lemma \ref{Hadamard}).
\end{proof}

Note that Corollary \ref{ind cor} (combined with Lemma \ref{LemNonIntBound}(2) to ensure that a Gram matrix with $m$, $2m$, $C_1m$, $D_1m$ on the diagonal has all entries in $\Z$) actually shows the following curious variation on the lifting problem \cite{KY1}: \textit{There are only finitely many real quadratic fields admitting a ternary lattice that represents all of $m\N$.} For $m=1$ this was present, although not explicitly stated, already in \cite{KKP}. We now get this result for all values of $m$ and with explicit bounds on the discriminant $\Delta_D$ of $\Q(\!\sqrt{D})$.

To be specific: Using the previous Corollary and Lemma \ref{LemNonIntBound}(2), one easily sees that any quadratic field admitting a classical totally positive lattice which represents $m\N$ must have $\Delta_D \leq 8m^5\gamma_p^2\gamma_q$ where $p \leq 2m^2$ and $q \leq 2\gamma_p m^3$. For $m=1$ and $m=2$ one gets better bounds on $C_1$ and $D_1$ by checking all possible pairs of $p,q$ just as in Section \ref{sec:m-uni}, leading to the following statement:

\begin{remark}
If $F = \Q(\!\sqrt{D})$ admits a classical (non-classical, resp.) totally positive definite ternary lattice which represents all of $\N$, then $\Delta_D\leq 420$ ($\Delta_D\leq 39760$, resp.).
\end{remark}

\section{Ranks 4 and 5}\label{sec:45}

We are now finally ready to prove Theorem \ref{ThmMain1} case by case. Recall that we use $\Gr{k}$ to denote the Gram matrix $\begin{pmatrix} \B(v_i,v_j) \end{pmatrix}_{1 \leq i,j \leq k}$ whenever the vectors $v_1, \ldots, v_k$ are given.

\begin{thm}\label{rank 3}
Let $m, D\in\N$, $D \neq 1$ squarefree, and $F=\Q(\!\sqrt{D})$. Let $L$ be a classical totally positive definite quadratic $\co_F$-lattice which represents every element of $m\co_F^{+}$. Then there exists a prime $p_{1}\leq 2m^{2}$ and an integer $2 \leq C_{1}\leq \gamma_{p_{1}}$ such that if $\Delta_{D}>4C_{1}m^{2}\big(\omega_{D}-\lfloor \omega_{D}'\rfloor\big)$, then $\rank L \ge 4$.
\end{thm}

\begin{proof}
Let $C_1 \geq 2$ be an integer for which the inequality $\Delta_D > 4C_1 m^2 \big(\omega_{D}-\lfloor \omega_{D}'\rfloor\big)$ holds. The specific value of $C_1$ will be chosen later. As $\omega_{D}-\lfloor \omega_{D}'\rfloor = \sqrt{\Delta_D} + \omega_{D}'-\lfloor \omega_{D}'\rfloor > \sqrt{\Delta_D} \geq \sqrt5$, it follows that $C_1$ and $\omega_{D}-\lfloor \omega_{D}'\rfloor$ are the largest two elements in the list $1,2,C_1,\omega_{D}-\lfloor \omega_{D}'\rfloor$. By the assumption that $L$ represents all of $m\co_F^+$, there exist vectors $v_1$, $v_2$, $v_3$, $v_4 \in L$ such that $Q(v_1) = m$, $Q(v_2) = 2m$, $Q(v_3) = C_1 m$ and $Q(v_4) = \big(\omega_{D}-\lfloor \omega_{D}'\rfloor\big)m$. We have already noted that $Q(v_3)Q(v_4) \geq Q(v_i)Q(v_j)$ for any choice $i \neq j$. Thus our condition on the discriminant yields
\[
\Delta_D/4 > C_1 m \cdot \big(\omega_{D}-\lfloor \omega_{D}'\rfloor\big)m = Q(v_3)Q(v_4) \geq Q(v_i)Q(v_j) \qquad \text{for all $i \neq j$ with $1 \leq i, j \leq 4$.}
\]
We also note that $Q(v_4) > Q(v_4)'$ and $Q(v_i) = Q(v_i)'$ for $1\leq i \leq 3$, so $\Delta_D/4 > Q(v_i)'Q(v_j)'$ in all cases as well. All in all, we have $Q(v_i)Q(v_j) \prec \Delta_D/4$ for $i \neq j$ and thus Lemma \ref{LemNonIntBound}(2) implies $\B(v_i,v_j) \in \Z$. That is, all off-diagonal entries in the Gram matrix $\Gr4$ are integers.

The vectors $v_1$, $v_2$ are linearly independent (unless $D=2$, in which case we can just put $p_{1}=C_1=2$ and then $\Delta_D > 4C_1m^2(\sqrt2 + 2)$ will never hold). Hence the assumptions of Corollary \ref{ind cor} are satisfied. This gives us a specific value of $p_{1} \leq 2m^2$ and $2 \leq C_{1} \leq \gamma_{p_{1}}$, for which we now prove the statement. First of all, Corollary \ref{ind cor}(1) immediately yields that $v_1, v_2, v_3$ are linearly independent. This means that $\det \Gr3\neq 0$.

Since $Q(v_4) = \big(\omega_{D}-\lfloor \omega_{D}'\rfloor\big)m$ is the only irrational entry of the matrix $\Gr4$, the Laplace expansion of the determinant yields
\[
\det \Gr4 =  \substack{\textrm{a rational}\\\textrm{number} } + \omega_{D} m  \det \Gr3.
\]
As $\det \Gr3$ is nonzero, it follows that $\det \Gr4$ is irrational and thus nonzero. This means by Lemma \ref{gram} that the vectors $v_1,v_2,v_3,v_4$ are linearly independent. Hence, indeed, $\rank L\geq 4$.
\end{proof}

In the subsequent proofs, we will repeatedly need the following observation about real matrices.

\begin{lem}\label{SumOfPSD}
Let $X$ and $Y$ be positive semidefinite matrices of the same size. Then for the positive semidefinite matrix $X+Y$, the following are equivalent:
 \begin{enumerate}
     \item $X+Y$ is positive definite,
     \item $\det(X+Y) \neq 0$,
     \item $\Ker X \cap \Ker Y = \{0\}$.
 \end{enumerate}
\end{lem}

We will use it in the following form.

\begin{cor} \label{SumOfBlockPSD}
Let $X, Y \in M_n(\R)$ be positive semidefinite matrices of the block form
\[
X =
\begin{pmatrix}
A & C \\ C^{\mathrm{T}} & D    
\end{pmatrix}, \qquad
Y =
\begin{pmatrix}
O & O \\ O & B    
\end{pmatrix},
\]
where the matrices $A \in M_k(\R)$ and $B \in M_{n-k}(\R)$ are positive definite and $O$ denotes a suitable zero matrix. Then $X+Y$ is a positive definite matrix.
\end{cor}
\begin{proof}
By the previous lemma, $X+Y$ is positive semidefinite, so it is enough to check that $\Ker X \cap \Ker Y = \{0\}$. Clearly, $\Ker Y = \R e_1 + \cdots + \R e_k$ where $e_i$ stands for the $i$th vector of the standard basis. Thus an element of $\Ker Y$ has the block form $v = \bigl(\begin{smallmatrix} x \\ 0 \end{smallmatrix}\bigr)$ where $x \in \R^k$. However, if also $v \in \Ker X$, then $x^{\mathrm{T}}Ax = 0$, which implies $x=0$ and thus $v=0$, as we need.
\end{proof}

\begin{thm} \label{rank 4}
Let $m, D\in\N$, $D \neq 1$ squarefree, and $F=\Q(\!\sqrt{D})$. Let $L$ be a classical totally positive definite quadratic $\co_F$-lattice which represents every element of $m\co_F^{+}$. Then there exists a prime $p_{1}\leq 2m^{2}$ and an integer $2 \leq C_{1}\leq \gamma_{p_{1}}$ such that if
$\Delta_{D}> 4C_{1}m^{2}\bigl(2\omega_{D}-\lfloor 2\omega_{D}' \rfloor\bigr),$
then $\rank L \ge 5$.
\end{thm}

\begin{proof}
First we note that the assumptions of Theorem \ref{rank 3} are satisfied, since 
\[
2\omega_{D}-\lfloor 2\omega_{D}'\rfloor = 2\sqrt{\Delta_D} + (2\omega_{D}'-\lfloor 2\omega_{D}'\rfloor) > 2\sqrt{\Delta_D} > \sqrt{\Delta_D} + 1 = \omega_{D}- \omega_{D}' + 1 > \omega_{D}-\lfloor \omega_{D}' \rfloor.
\]
Therefore, by performing the same process as in the proof of Theorem \ref{rank 3}, we obtain the required $p_1, C_1$ as well as linearly independent vectors $v_{1},v_{2},v_{3},v_{4}\in L$ such that $Q(v_{1})=m$, $Q(v_{2})=2m$, $Q(v_{3})=C_{1}m$, $Q(v_{4})=(\omega_{D}-\lfloor\omega_{D}'\rfloor )m$ and $\B(v_{i},v_{j})\in\Z$ unless $i=j=4$.

Let us pick a vector $v_{5}\in L$ so that $Q(v_{5})=(2\omega_{D}-\lfloor 2\omega_{D}'\rfloor )m\in m\co_{F}^{+}$. We will show that $\det \Gr5 \neq 0$. Lemma \ref{gram} will then ensure that the vectors $v_{1},v_{2},v_{3},v_{4},v_{5}$ are linearly independent, and so $\rank L\geq 5$.

We shall often need the inequalities $0<Q(v_{4})'<m$, $0<Q(v_{5})' <m$. These follow from trivial facts about the floor function, as
\[
Q(v_4)' = \left(\omega_{D}'-\lfloor\omega_{D}'\rfloor\right)m \quad \textrm{ and } \quad 
Q(v_5)' = \left(2\omega_{D}'-\lfloor 2\omega_{D}'\rfloor\right)m.
\]

In the proof of Theorem \ref{rank 3}, we have already shown $\B(v_{i},v_{j})\in\Z$ for $1\leq i,j\leq 4$ unless $i=j=4$. In the same way, Lemma \ref{LemNonIntBound}(2) implies that if $1\leq i\leq 3$ and $j=5$, then again $\B(v_{i},v_{j})\in\Z$. Hence only for $i,j \in \{4,5\}$ the element may lie outside of $\Z$.

Our aim will be to decompose the matrix $\Gr5$ as a sum of two positive semidefinite matrices and then apply Corollary \ref{SumOfBlockPSD}. With this purpose in mind, we shall write
\[
\B(v_{i},v_{j})=a_{ij}+b_{ij}\sqrt{\Delta_{D}},
\]
where we put $a_{ij}:= \B(v_{i},v_{j})'$, which then uniquely defines the elements $b_{ij}$. First we observe that in fact, $b_{ij} \in \Z$. Indeed, if $a$ and $b$ are integers such that $\B(v_{i},v_{j})=a+b\omega_{D}$, then
\[
\B(v_{i},v_{j})=a+b\omega_{D} = (a+b\omega_{D}') + b(\omega_{D}-\omega_{D}') = \B(v_{i},v_{j})' + b\sqrt{\Delta_{D}} = a_{ij} + b\sqrt{\Delta_{D}}.
\]
Let us also note that the matrices $A=(a_{ij})$ and $B=(b_{ij})$ are by definition symmetric and that the former is positive semidefinite, as $A' = \Gr5$, which is totally positive semidefinite. We also know that unless both $i,j \in \{4,5\}$, then $a_{ij} \in \Z$ and $b_{ij}=0$.

The positive definiteness of $L$ implies that if $1\leq i\leq 3$ and $4\leq j \leq 5$, then
\begin{align*}
a_{ij}^{2}=\left(\B(v_{i},v_{j})'\right)^{2}\leq Q(v_{i})'Q(v_{j})' = Q(v_{i})Q(v_{j})'<Q(v_{i})m.
\end{align*}
Further, our inequalities about $Q(v_4)'$ and $Q(v_5)'$ can be written as $a_{44}, a_{55} \in (0,m)$. One also easily sees that $b_{44}=m$ and $b_{55}=2m$.

It remains to deal with $\B(v_{4},v_{5})$. By positive definiteness of $L$ we get $|a_{45}|<m$, since
\[
a_{45}^{2}=\big(\B(v_{4},v_{5})'\big)^{2}\leq Q(v_{4})'Q(v_{5})'<m^{2}.
\]

Now we will show that $|b_{45}| < \sqrt{2}m$. Suppose that this is not true. Then $2m^{2}-b_{45}^{2} \leq 0$. These are all integers and equality is impossible, therefore $2m^{2}-b_{45}^{2}\leq -1$, and hence $|b_{45}|\ge \sqrt{2m^2+1}$. We already know that $|a_{45}|<m$, so we obtain $\B(v_4,v_5)^2=(a_{45}+b_{45}\sqrt{\Delta_D})^2>(\sqrt{(2m^2+1)\Delta_D}-m)^2$. Using the identity for $\omega_D-\omega_D'$, we get that $Q(v_4)Q(v_5)<(\sqrt{\Delta_D}+1)(2\sqrt{\Delta_D}+1)m^2$. This implies
\begin{align*}
    Q(v_4)Q(v_5)&\ge \B(v_4,v_5)^2\\
    (\sqrt{\Delta_D}+1)(2\sqrt{\Delta_D}+1)m^2&>\bigl(\sqrt{(2+1/m^2)\Delta_D}-1\bigr)^2m^2\\
    (\sqrt{\Delta_D}+1)(2\sqrt{\Delta_D}+1)&>(2+1/m^2)\Delta_D-2\sqrt{(2+1/m^2)\Delta_D}+1\\
    m^2(3+2\sqrt{2+1/m^2})&>\sqrt{\Delta_D}.
\end{align*}
However, this is in contradiction to the inequality $\sqrt{\Delta_D} > 16m^2$, which easily follows from the assumptions of the theorem by the computation $\Delta_D > 4C_1 m^2 \bigl(2\omega_{D}-\lfloor 2\omega_{D}' \rfloor\bigr) > 4\cdot2m^2\cdot 2\sqrt{\Delta_D}$. Thus, indeed, $|b_{45}| < \sqrt{2}m$.

Let us use the information we have obtained. We have shown that
\[
\Gr5=A+B\sqrt{\Delta_D},
\]
where $A=(a_{ij})_{1\leq i,j\leq 5}$, $B=(b_{ij})_{1\leq i,j\leq 5}$. Note that $b_{ij}=0$ for $i<4$ or $j<4.$
From the previous observations we have $b_{45}^{2} < 2m^{2}=b_{44}b_{55}$, so the Sylvester's criterion (see, for example, \cite[Theorem 7.2.5]{HJ}) shows that the $2 \times 2$ matrix $(b_{ij})_{4\leq i,j\leq 5}$ is positive definite. Thus the matrix $B$ is positive semidefinite, and hence so is the matrix $B\sqrt{\Delta_{D}}$ -- however, note that we do not claim it to be \emph{totally} positive semidefinite; this is not the case, as $\Delta_D'<0$. As already noted, $A=(\B(v_{i},v_{j})')_{1\leq i,j\leq 5}$, so it is also positive semidefinite. Using Corollary \ref{SumOfBlockPSD}, we see that $\Gr5 = A+B\sqrt{\Delta_D}$ is positive definite: Both $5 \times 5$ matrices are positive semidefinite; the $2\times 2$ submatrix of $B\sqrt{\Delta_D}$ is clearly positive definite; the $3 \times 3$ matrix $(a_{ij})_{1\leq i,j\leq 3}$ is also positive definite, as it is equal to $\Gr3'$ and we already know that $v_1,v_2,v_3$ are linearly independent, making $\Gr3$ totally positive definite.

Once we know that $\Gr5$ is positive definite, the result follows from Lemma \ref{gram}.
\end{proof}

\section{Ranks 6, 7, and 8}\label{sec:678}

Among the ingredients needed in the next part of the proof is the following statement, closely related to Sylvester's criterion.

\begin{lem} \label{SylvesterPSD}
    Let $M\in M_{n}(\R)$ be a symmetric matrix of the form
    \begin{align*}
        M=\begin{pmatrix}
            A & v \\
            v^{\mathrm{T}} & a
        \end{pmatrix},
    \end{align*}
    with $A\in M_{n-1}(\R)$ positive definite, $v \in \R^{n-1}$, $a \in \R$. If $\det M\geq 0$, then $M$ is positive semidefinite.
\end{lem}
\begin{proof}
    Let us take any $w \in \R^n$; it is enough to show that $w^{\mathrm{T}}Mw\geq 0$. We write $w=\bigl(\!\begin{smallmatrix}
        \tilde{w} \\ w_n
    \end{smallmatrix}\!\bigr)$, where $\tilde{w}\in\R^{n-1}$ and $w_{n}$ is the $n$th coordinate of $w$. Then
     \[
        w^{\mathrm{T}}Mw = \tilde{w}^{\mathrm{T}}A\tilde{w} + w_n\tilde{w}^{\mathrm{T}}v + w_n v^{\mathrm{T}}\tilde{w} + aw_{n}^{2} 
        = \left(\tilde{w}+u\right)^{\mathrm{T}}A\left(\tilde{w}+u\right) + w_{n}^{2}\left(a - v^{\mathrm{T}}A^{-1}v\right),
    \]
    where $u=w_{n}A^{-1}v$. We assume $A$ positive definite so $(\tilde{w}+u)^{\mathrm{T}}A(\tilde{w}+u)\geq 0$. Therefore, it is enough to prove that $a - v^{\mathrm{T}}A^{-1}v\geq 0$. Recall the block matrix determinant formula (see, e.g., \cite[(0.8.5.1)]{HJ}): If $X$ and $T$ are square matrices, $\det X\neq 0$, and $Y$, $Z$ are matrices of compatible size, then
    \begin{align*}
        \det \begin{pmatrix}
            X & Y \\
            Z & T
        \end{pmatrix} = \det X \det\left(T - ZX^{-1}Y\right).
    \end{align*}    
    Using the above formula for $X=A$, $Y=v$, $Z=v^{\mathrm{T}}$ and $T=a$, we obtain:
    \begin{align*}
        0\leq \det M = \det A \cdot\left(a - v^{\mathrm{T}}A^{-1}v\right).
    \end{align*}
    Since $\det A>0$, we indeed have $a - v^{\mathrm{T}}A^{-1}v\geq 0$ and the result follows.
\end{proof}

\begin{thm} \label{rank 5}
Let $m, D\in\N$, $D \neq 1$ squarefree, and $F=\Q(\!\sqrt{D})$. Let $L$ be a classical totally positive definite quadratic $\co_F$-lattice which represents every element of $m\co_F^{+}$. Then there exist primes $p_{1}, p_{2} \leq 2m^{2}$ and integers $2 \leq C_{i}\leq \gamma_{p_{i}}$ such that if
\begin{enumerate}
\item $\Delta_{D}> 4C_{1}m^{2} \bigl(C_{2}\omega_{D} - \lfloor C_{2}\omega_{D}'\rfloor\bigr)$ and
\item $\sqrt{\Delta_D}^3 > 36C_{2}m^{3} \sqrt{\Delta_{D}}^{2} + 18C_2m^{3} \sqrt{\Delta_{D}} + m^{3}$,
\end{enumerate}
then $\rank L \ge 6$.
\end{thm}

\begin{proof}
First of all, analogously to the beginning of the proof of Theorem \ref{rank 4}, we observe that for any integer $C_2 \geq 2$, we get $C_{2}\omega_{D} - \lfloor C_{2}\omega_{D}'\rfloor \geq 2\omega_{D}-\lfloor 2\omega_{D}'\rfloor$, so the assumptions of Theorem \ref{rank 4} are satisfied and we can exploit all its conclusions.

Let $v_{1},v_{2},v_{3},v_{4},v_{5}\in L$ be vectors defined in the proof of Theorem \ref{rank 4}, and $p_{1}\leq 2m^{2}$ and $2\leq C_{1}\leq \gamma_{p_{1}}$ be the appropriate numbers from the same proof. For all $1\leq i,j\leq 5$ let the numbers $a_{ij} \in \co_F$ and $b_{ij} \in \Z$ satisfying $\B(v_{i},v_{j})=a_{ij}+b_{ij}\sqrt{\Delta_{D}}$ be the same as in the proof of Theorem \ref{rank 4}. We shall need most of the information about them obtained therein.

By the proof of Theorem \ref{rank 4}, the matrix $(b_{ij})_{4\leq i,j\leq 5} \in M_2(\Z)$ is positive definite, so Corollary \ref{C2cor} applies. This gives us the desired prime $p_{2} \leq b_{44}b_{55} = 2m^2$ and integer $2 \leq C_2 \leq \gamma_{p_{2}}$. We use it to define the next vector in our to-be-proved-linearly-independent set.

Consider any vector $v_{6}\in L$ such that $Q(v_{6})=m\left(C_{2}\omega_{D}-\lfloor C_{2}\omega_{D}'\rfloor\right)$. Let us again write $\B(v_i,v_6)=a_{i6}+b_{i6}\sqrt{\Delta_D}$ where $a_{i6}=\B(v_i,v_6)'$. Then $b_{66}=C_{2}m$, and by performing the same computations as in the proof of Theorem \ref{rank 4}, we get $a_{66}\in (0,m)$. We again have the decomposition 
\[
\Gr6=A+B\sqrt{\Delta_D},
\]
where $A=(a_{ij})_{1\leq i,j\leq 6}$, $B=(b_{ij})_{1\leq i,j\leq 6}$. We know that $A=\Gr6'$ is positive semidefinite and that $B \in M_6(\Z)$. As before, by Lemma \ref{LemNonIntBound}(2) we see that $\B(v_{i},v_{6})\in\Z$ for $i\leq 3$. Thus $b_{16}=b_{26}=b_{36}=0$, so the matrix $B$ has the block form
\[
B = 
\begin{pmatrix}
O & O \\ O & \widetilde{B}    
\end{pmatrix}, \qquad \text{where } \widetilde{B}= \begin{pmatrix} b_{ij} \end{pmatrix}_{4 \le i, j \le 6} \in M_3(\Z).
\]

Next, we can show that $|b_{46}| \leq \sqrt{C_2}m$. The approach is analogous to the one for proving $|b_{45}| < \sqrt{2}m$ in the proof of Theorem \ref{rank 4}: One assumes the contrary, which gives $|b_{46}| \geq \sqrt{C_2m^2+1}$, and in several steps one deduces $m^2 (C_2 + 2\sqrt{C_2+1/m^2}+1) \geq \sqrt{\Delta_D}$. However, this is a contradiction, as our assumptions imply $\sqrt{\Delta_D}>8C_2m^2$. In the same way, we can show $|b_{56}| \leq \sqrt{2C_2}m$.

Our next goal is to prove that $\det \widetilde{B} \geq 0$. To get a contradiction, suppose that $\det \widetilde{B} \leq -1$. Let us consider the determinant of the positive semidefinite matrix $\Gra46 = \begin{pmatrix} \B(v_i,v_j) \end{pmatrix}_{4 \le i,j \le 6}$:
\[
\det \Gra46 =
\det \begin{pmatrix} 
b_{44} \sqrt{\Delta_D}+a_{44} & b_{45}\sqrt{\Delta_D}+a_{45} & b_{46}\sqrt{\Delta_D}+a_{46} \\
b_{54}\sqrt{\Delta_D}+a_{54} & b_{55}\sqrt{\Delta_D}+a_{55} & b_{56}\sqrt{\Delta_D}+a_{56} \\
b_{64}\sqrt{\Delta_D}+a_{54} & b_{65}\sqrt{\Delta_D}+a_{65} & b_{66}\sqrt{\Delta_D}+a_{66} \\
\end{pmatrix}.
\] 
Let $\widetilde{A}=(a_{ij})_{4 \le i, j\le 6}$, and denote by $\widetilde{A}_{ij}$ (or $\widetilde{B}_{ij}$) the $2 \times 2$ submatrix of $\widetilde{A}$ (or $\widetilde{B}$) obtained by omitting the $i$th row and $j$th column. Then $\det \Gra46$ can be expanded as a polynomial in $\sqrt{\Delta_{D}}$ (note that this is also a special case of Lemma \ref{detA+B} below):
\[
\det \Gra46 = \det \widetilde{B} \cdot\sqrt{\Delta_D}^3+\sum_{4 \le i,j \le 6}(-1)^{i+j}a_{ij} \det \widetilde{B}_{ij} \cdot\sqrt{\Delta_D}^2+ \sum_{4 \le i, j \le 6}(-1)^{i+j}b_{ij}\det \widetilde{A}_{ij} \cdot\sqrt{\Delta_D}+\det \widetilde{A}.
\]

From the construction we know that $|a_{ij}|<m$ for all $4\leq i,j\leq 6$. This implies $\mathopen|\det \widetilde{A}_{ij}\mathclose|<2m^{2}$ for all $4\leq i,j\leq 6$. Moreover, $\widetilde{A}$ is positive semidefinite, so $\det\widetilde{A}< m^{3}$ by Hadamard's inequality (Lemma \ref{Hadamard}). Our bounds on the elements of $\widetilde{B}$ easily yield $\mathopen|\det\widetilde{B}_{ij}\mathclose|\leq 4C_{2}m^{2}$ for all $4\leq i,j\leq 6$.
Putting all of these bounds together leads to
\begin{align*}
\det \Gra46  & \leq -\sqrt{\Delta_{D}}^{3} + 9\cdot m\cdot 4C_{2}m^{2} \sqrt{\Delta_{D}}^{2} + 9\cdot C_2m \cdot 2m^{2} \sqrt{\Delta_{D}} + m^{3} \\
& = -\sqrt{\Delta_{D}}^{3}  + 36C_{2}m^{3} \sqrt{\Delta_{D}}^{2} + 18C_2m^{3} \sqrt{\Delta_{D}} + m^{3} <0,
\end{align*}
where the last inequality follows from the assumptions. This is a contradiction, since as a Gram matrix, $\Gra46$ is positive semidefinite. Therefore, we have $\det \widetilde{B}\geq 0$.

Finally we can exploit our choice of $C_2$. We have checked the assumptions of Lemma \ref{SylvesterPSD}, so $\widetilde{B}$ is positive semidefinite. We see that it is of the form required in Corollary \ref{C2cor}. Therefore, $\widetilde{B}$ is in fact positive definite. 

Now we can conclude the proof in exactly the same way as in the proof of Theorem \ref{rank 4}: Clearly, $\widetilde{B}\sqrt{\Delta_D}$ is also positive definite, and we already know that $A = (a_{ij})_{1 \leq i,j \leq 6} = \Gr{6}'$ is positive semidefinite and its submatrix $A = (a_{ij})_{1 \leq i,j \leq 3} = \Gr{3}'$ is positive definite, as $v_1,v_2,v_3$ are linearly independent. Thus we are exactly in the situation handled in Corollary \ref{SumOfBlockPSD}: Our matrices $A$ and $B\sqrt{\Delta_D}$ are positive semidefinite, the latter is a block matrix formed by a $3 \times 3$ positive definite matrix padded by zeros, and the former has a positive definite upper left corner. Thus, Corollary \ref{SumOfBlockPSD} ensures that $\Gr6 = A + B\sqrt{\Delta_D}$ is positive definite. By Lemma \ref{gram}, vectors $v_{1},\ldots,v_{6}$ are linearly independent. This completes our claim.
\end{proof}

\begin{thm} \label{rank 6}
Let $m, D\in\N$, $D \neq 1$ squarefree, and $F=\Q(\!\sqrt{D})$. Let $L$ be a classical totally positive definite quadratic $\co_F$-lattice which represents every element of $m\co_F^{+}$. Then there exist primes $p_{1},p_{2}\leq 2m^{2}$ and integers $2 \leq C_{i}\leq \gamma_{p_{i}}$, and also a prime $q_{1}\leq 2C_{1}m^{3}$ and an integer $1\leq D_{1}\leq q_{1}\gamma_{q_{1}}$ such that if
\begin{enumerate}
\item $\Delta_{D}> 4C_{1}m^{2} \bigl(C_{2}\omega_{D} - \lfloor C_{2}\omega_{D}'\rfloor\bigr)$,
\item $\sqrt{\Delta_D}^3 > 36C_{2}m^{3} \sqrt{\Delta_{D}}^{2} + 18C_2m^{3} \sqrt{\Delta_{D}} + m^{3}$, and
\item $\Delta_{D}> 4D_{1}m^{2} \max\bigl\{C_1,C_{2}\omega_{D} - \lfloor C_{2}\omega_{D}'\rfloor\bigr\}$,
\end{enumerate}
then $\rank L \ge 7$.
\end{thm}
\begin{proof}
Let the numbers $p_{1}$, $C_{1}$, $p_{2}$, $C_{2}$ and vectors $v_{1},\ldots ,v_{6}$ be the same as in the proof of Theorem \ref{rank 5}. Let us take the prime $q_{1}$ and the $D_{1} \in \Z^+$ from Corollary \ref{ind cor}(2). There exists a vector $v_{0}\in L$ such that $Q(v_{0})=D_{1}m$. (It would be more natural to call this vector $v_7$, but denoting it $v_0$ permutes the entries of the Gram matrix, simplifying its block structure.) 

Let us write $\B(v_{i},v_{j})=a_{ij}+b_{ij}\sqrt{\Delta_{D}}$ as before. This gives the decomposition $\Gra{0}{6} = A + B \sqrt{\Delta_D}$. For every $1 \leq i \leq 6$, we compute that $Q(v_i)Q(v_0) \prec \Delta_D/4$ thanks to assumption (3), so Lemma \ref{LemNonIntBound}(2) yields $\B(v_i,v_0) \in \Z$ and thus $b_{i0}=0$. Thus the matrix $B$ (and $B \sqrt{\Delta_D}$) still has the block structure with positive definite $3 \times 3$ part in the bottom right corner. The matrix $A = \Gra06'$ is positive semidefinite. Due to the choice of $D_1$, the vectors $v_1,v_2,v_3,v_0$ are independent by Corollary \ref{ind cor}, so the $4 \times 4$ matrix $(a_{ij})_{0 \leq i,j \leq 3} = \Gra03'$ is positive definite. Thus the situation is exactly prepared so that Corollary \ref{SumOfBlockPSD} yields the positive definiteness of $\Gra06 = A + B\sqrt{\Delta_D}$, ensuring the independence of $v_0, v_1, \ldots, v_6$.
\end{proof}

For the final part of the proof, we also need the $4 \times 4$ case of the following formula for the determinant of the sum of two matrices.

\begin{lem}[\cite{Mar}] \label{detA+B}
Let $M,N \in M_n(R)$ for any commutative ring $R$. Then
\[
\det(M+N) = \sum_{r=0}^n \sum_{\alpha \in \Xi_r}\sum_{\beta \in \Xi_r} (-1)^{s(\alpha)+s(\beta)} \det M_{\alpha}^{\beta} \det N_{\overline{\alpha}}^{\overline{\beta}},
\]
where the following notation is used: $\Xi_r$ is the set of all subsequences of $1, \ldots, n$ of length $r$; $s(\alpha)$ is the sum of all integers in the sequence $\alpha$; $M_{\alpha}^{\beta} \in M_{n-r}(R)$ is obtained from $M$ by omitting rows with indices in $\alpha$ and columns with indices in $\beta$; $N_{\overline{\alpha}}^{\overline{\beta}} \in M_r(R)$ is obtained from $N$ by only keeping rows with indices in $\alpha$ and columns with indices in $\beta$.
\end{lem}

We use it in the following form:

\begin{lem} \label{betterDetA+B}
Let $\widetilde{A}, \widetilde{B} \in M_4(\R)$, $\delta > 0$.
For $1 \leq i,j \leq 3$, let $a_i, b_j \in \R$ be such that for any $i \times i$ minor $\det X$ of $\widetilde{A}$ (that is, a determinant of any $i \times i$ submatrix $X$ of $\widetilde{A}$), and, respectively, any $j\times j$ minor $\det Y$ of $\widetilde{B}$, it holds that $\mathopen| \det X \mathclose| \leq a_i$ and $\mathopen|\det Y\mathclose|\le b_j$. 
Then
\[
\det(\widetilde{B}\delta + \widetilde{A}) \leq \det(\widetilde{B}) \delta^{4} + 16b_3 a_1 \delta^{3} + 36 b_2a_2 \delta^{2}  + 16 b_1a_3\delta + \det\widetilde{A}.
\]
\end{lem}
\begin{proof}
This is a direct application of Lemma \ref{detA+B}: The five summands correspond to $r=0, 1, 2, 3, 4$; we keep the first and last summand, while we bound all the other in an obvious way. The constants correspond to the number of terms in each of these summands, e.g.\ $36 = \binom{4}{2}\binom{4}{2}$.
\end{proof}

\begin{thm} \label{rank 7}
Let $m, D\in\N$, $D \neq 1$ squarefree, and $F=\Q(\!\sqrt{D})$. Let $L$ be a classical totally positive definite quadratic $\co_F$-lattice which represents every element of $m\co_F^{+}$. Then there exist primes $p_{1},p_{2}\leq 2m^{2}$ and integers $C_1, C_2$ satisfying $2 \leq C_{i}\leq \gamma_{p_{i}}$, and also primes $q_1, q_2$ with $q_{i}\leq 2C_{i}m^{3}$ and integers $D_1, D_2$ satisfying $1\leq D_{i}\leq q_{i}\gamma_{q_{i}}$ such that if
\begin{enumerate}
\item $\Delta_{D}> 4C_{1}D_{1} m^{2}$,
\item $\Delta_{D}> 4\max\{C_1,D_1\} m^{2} \max\bigl\{C_{2}\omega_{D} - \lfloor C_{2}\omega_{D}'\rfloor, D_{2}\omega_{D} - \lfloor D_{2}\omega_{D}'\rfloor \bigr\}$, \label{it2}
\item $\sqrt{\Delta_D}^3 > 36C_{2}m^{3} \sqrt{\Delta_{D}}^{2} + 18C_2m^{3} \sqrt{\Delta_{D}} + m^{3}$, and
\item $\sqrt{\Delta_{D}}^{4} > 192C_{2}D_{2}m^{4} \sqrt{\Delta_{D}}^{3} + 144C_{2}D_{2}m^{4}\sqrt{\Delta_{D}}^{2} + 96\max\{C_{2},D_{2}\}m^{4}\sqrt{\Delta_{D}} + m^{4}$,
\end{enumerate}
then $\rank L \ge 8$.
\end{thm}
\begin{proof}
First of all, note that the assumptions of Theorem \ref{rank 6} are satisfied, so we can use all its conclusions and also all inequalities obtained during its proof. We let the primes $p_1, p_2, q_1$, the numbers $C_1, C_2, D_1$ and the vectors $v_0, \ldots, v_6$ be the same as in the proof of Theorem \ref{rank 6}. As before, we write $\B(v_{i},v_{j})=a_{ij}+b_{ij}\sqrt{\Delta_{D}}$. From the previous proofs, we know that the $3 \times 3$ integer matrix $(b_{ij})_{4 \leq i,j \leq 6}$ is positive definite. Therefore we can apply Corollary \ref{D2cor} to this matrix to obtain a prime $q_2$ and the corresponding integer $D_2$. (For future use, observe that $D_2 \geq 2$, since the Corollary can never be satisfied with $D_2=1$.) Now, let $v_7\in L$ be a vector such that $Q(v_7)=m(D_{2}\omega_{D}-\lfloor D_{2}\omega_{D}'\rfloor)$. We extend our notation and write $\B(v_{i},v_7)=a_{i7}+b_{i7}\sqrt{\Delta_{D}}$ just as before; thus, we once again have the decomposition
\[
\Gra07=A+B\sqrt{\Delta_D},
\]
where $A=(a_{ij})_{0\leq i,j\leq 7}$, $B=(b_{ij})_{0\leq i,j\leq 7}$. We know that $A=\Gra07'$ is positive semidefinite and that $B \in M_8(\Z)$. As before, by Lemma \ref{LemNonIntBound}(2) we see that $\B(v_{i},v_{7})\in\Z$ for $0\leq i \leq 3$ thanks to assumption \eqref{it2}. Thus $b_{07}=\cdots =b_{37}=0$, so the matrix $B$ has the block form
\[
B = 
\begin{pmatrix}
O & O \\ O & \widetilde{B}    
\end{pmatrix}, \qquad \text{where } \widetilde{B}= \begin{pmatrix} b_{ij} \end{pmatrix}_{4 \le i, j \le 7} \in M_4(\Z).
\]

Now, let us prepare an auxiliary inequality: We claim that the assumption \eqref{it2} implies
\begin{equation} \label{inequality}
\sqrt{\Delta_D} > m^2 \bigl(2M + 2\sqrt{M^2 + 1/{m^2}}\bigr) \quad \text{where } M = \max\{C_2,D_2\}.
\end{equation}
It is clear that it is enough to show $\sqrt{\Delta_D} > 2m^2\left(M+\sqrt{M^2+M^2} \right)= 2(1+\sqrt2)m^2M$. On the other hand, \eqref{it2} directly yields $\Delta_D > 4 \cdot 2m^2 \max\{C_2 \sqrt{\Delta_D}, D_2 \sqrt{\Delta_D}\}$, which can be rewritten as $\sqrt{\Delta_D} > 8m^2 M$. This implies the desired inequality and thus proves our claim.

Our next goal is to prove that $\det \widetilde{B} \geq 0$. For this purpose, let us consider the determinant of the positive semidefinite matrix $\Gra47$:
\[
\det\Gra47 = \det(\widetilde{B}\sqrt{\Delta_{D}} + \widetilde{A}) \leq \det(\widetilde{B}) \sqrt{\Delta_{D}}^{4} + 16b_3 a_1 \sqrt{\Delta_{D}}^{3} + 36 b_2a_2 \sqrt{\Delta_{D}}^{2}  + 16 b_1a_3\sqrt{\Delta_{D}} + \det\widetilde{A},
\]
by Lemma \ref{betterDetA+B}, where the $a_i$ are bounds for absolute values of $i \times i$ minors of $\widetilde{A}$ and similarly for the $b_i$. We need to compute these six constants.

We have $|a_{ij}| < m$ for all $4 \leq i,j \leq 7$, so we can put $a_1 = m$. Using a trivial bound for the determinant, we can also put $a_2 = 2m^2$ and $a_3 = 6m^3$. Moreover, $\widetilde{A}$ is positive semidefinite, so Hadamard's inequality (Lemma \ref{Hadamard}) implies $\det \widetilde{A} < m^4$. 

Next, we prove that $b_1 = \max\{C_2,D_2\}m$, $b_2 = 2C_2D_2m^2$ and $b_3 = 12C_2D_2m^3$ will suffice. Recall that we have $|b_{45}| < \sqrt{2}m$ by the proof of Theorem \ref{rank 4} and $|b_{46}| \leq \sqrt{C_2}m$ together with $|b_{56}| \leq \sqrt{2C_2}m$ by the proof of Theorem \ref{rank 5}; this can be summarized as $b_{ij}^2 \leq b_{ii}b_{jj}$ for $4 \leq i,j \leq 6$. By the same technique, we get that $b_{i7}^2 \leq b_{ii}b_{77}$ for $i=4,5,6$; it turns out that the inequality \eqref{inequality} is sufficient for this part of the proof to work. All in all, $b_{ij}^2 \leq b_{ii}b_{jj}$ for any pair $4 \leq i,j \leq 7$, which yields that for every choice of pairwise distinct numbers $i,j,k\in\{4,5,6,7\}$ and every permutation $\sigma$ of the set $\{4,5,6,7\}$ we have
\begin{align*}
b_{i\sigma (i)}^{2}b_{j\sigma (j)}^{2}b_{k\sigma (k)}^{2}\leq b_{ii}b_{\sigma (i)\sigma (i)}b_{jj}b_{\sigma (j)\sigma (j)}b_{kk}b_{\sigma (k)\sigma (k)}\leq b_{55}^{2}b_{66}^{2}b_{77}^{2}=\left(2C_{2}D_{2}m^{3}\right)^{2}.
\end{align*}
Thus, $\mathopen|b_{i\sigma (i)}b_{j\sigma (j)}b_{k\sigma (k)}\mathclose|\leq 2C_{2}D_{2}m^{3}$, which yields $b_3 = 12C_2D_2m^3$ as we claimed. Similarly we obtain $\mathopen|b_{i\sigma (i)}b_{j\sigma (j)}\mathclose|\leq C_{2}D_{2}m^{2}$, which gives $b_2 = 2C_2D_2m^2$. The bound $b_1 = \max\{C_2,D_2\}m$ is easy, since the largest element has to occur on the diagonal due to the inequalities $b_{ij}^2 \leq b_{ii}b_{jj}$.

Using the upper bound achieved above for the $\det \Gra47$, we assume for contradiction that $\det \widetilde{B} <0$, that is, $\det \widetilde{B} \leq -1$. Then
\begin{align*}
\det \Gra47 \leq &-\sqrt{\Delta_{D}}^{4} + 16\cdot 12C_{2}D_{2}m^{3}\cdot m \cdot \sqrt{\Delta_{D}}^{3} +\\
 & + 36\cdot 2C_{2}D_{2}m^{2}\cdot 2m^{2}\cdot \sqrt{\Delta_{D}}^{2} + 16\cdot \max\{C_{2},D_{2}\}m\cdot 6m^{3} \cdot \sqrt{\Delta_{D}} + m^{4}.
\end{align*}
The last expression is equal to
\begin{align*}
-\sqrt{\Delta_{D}}^{4} + 192C_{2}D_{2}m^{4} \sqrt{\Delta_{D}}^{3} + 144C_{2}D_{2}m^{4}\sqrt{\Delta_{D}}^{2} + 96\max\{C_{2},D_{2}\}m^{4}\sqrt{\Delta_{D}} + m^{4},
\end{align*}
which is negative by the assumption. This is a contradiction to the positive semidefiniteness of $\Gra47$. Therefore, $\det \widetilde{B} \geq 0$.

From the proof of Theorem \ref{rank 5}, we know that the first three leading principal minors of $\widetilde{B}$ are positive, and we just proved that the fourth is nonnegative, so Lemma \ref{SylvesterPSD} yields that $\widetilde{B}$ is positive semidefinite. Thus Corollary \ref{D2cor} applies and $\widetilde{B}$ is in fact positive definite. Thus, as in the previous proofs, we can apply Corollary \ref{SumOfBlockPSD} to get that $\Gra07 = A + B\sqrt{\Delta_D}$ is positive definite. This means linear independence of $v_0, \ldots, v_7$ and concludes the proof.
\end{proof}

\section{Explicit Kitaoka's Conjecture}\label{sec:ternary}
In this final section, we will identify all the real quadratic fields that may admit a universal ternary lattice. This generalizes the seminal work of Chan, Kim, and Raghavan \cite{CKR}, where they identified all quadratic fields with classical universal ternary quadratic forms. Specifically, we prove the following refinement of Theorem \ref{ThmMain2}, using the already established Theorem \ref{MAINTHEOREM1}.

\begin{thm}\label{thm:ternary}
If a real quadratic field $\Q(\!\sqrt{D})$ admits a universal ternary lattice, then $D$ is one of the thirteen values: $2,3,5,6,7,10,13,17,21,33,41,65,77$. Furthermore, for $D=10, 65$, no free universal ternary lattice exists.
\end{thm}

Based on computational evidence, we conjecture that the result above should be ``if and only if''. 

\begin{conjecture}\label{conj:list}
Let $D \in \{2,3,5,6,7,10,13,17,21,33,41,65,77\}$, $F=\Q(\!\sqrt{D})$. Then there exists a universal ternary $\co_F$-lattice. 
\end{conjecture}

We have concrete candidates for the universal lattices in these $13$ fields (see Theorem \ref{thm:list}). However, proving universality of a given quadratic lattice can be very difficult, and the one simple tool at our disposal works only for certain lattices. Namely, there is a notion of the \emph{class number} of a quadratic lattice; and if a lattice has class number $1$, then it satisfies an (integral) \emph{local-global principle} \cite[102:5]{OM}. In such a case, to prove universality of $L$, it is enough to study $L$ over all the local fields $F_{\mathfrak{p}}$ where $\mathfrak{p}$ is a prime ideal in $F$. Lattices over local fields are well understood, so if a lattice has class number $1$, which we easily compute in Magma, then proving its universality is routine work.

Thus, Conjecture \ref{conj:list} is proven only for $D=2,3,5,13,17$. In those fields (and only in them), one finds candidates for universal ternary lattices with class number $1$, and by local-global principle, universality follows (see Theorem \ref{thm:h=1}). If we require a free classical universal ternary $\co_F$-lattice, then only three such quadratic fields exist, namely $D=2,3,5$, as proved in \cite{CKR}. As a corollary of our results, we can strengthen this as follows.

\begin{thm} \label{CKRimproved}
A real quadratic field $\Q(\!\sqrt{D})$ admits a classical universal ternary lattice if and only if it admits a free classical universal ternary lattice; this holds precisely if $D = 2, 3, 5$.
\end{thm}

This is indeed an improvement, as the technique used in \cite{CKR} relied on the freeness of the lattices. Furthermore, since all the three fields in question have class number $1$, they only admit free lattices. This means that in fact, \emph{over real quadratic fields, there exist no non-free classical universal ternary lattices}. 

\smallskip

After this introduction, let us proceed to the proof of Theorem \ref{thm:ternary}. Take a field $F=\Q(\!\sqrt{D})$ equipped with a universal ternary lattice $(L,Q)$, possibly non-classical. Then the lattice $(L,2Q)$ is classical and represents all of $2\co_F^+$, so by Theorem \ref{MAINTHEOREM1}, we know that $\Delta_D < 10000$.
Therefore, it remains to check all the small eligible values of $D$, which we do using a simple computer program. We split the discussion into two cases, depending on the value of $D$ modulo 4. The Magma codes used for the proofs of Propositions \ref{pr:1mod4}, \ref{pr:23mod4} and Theorems \ref{thm:list}, \ref{thm:h=1} are available upon request. 

\begin{prop} \label{pr:1mod4}
Assume that $D \equiv 1 \pmod4$ is squarefree,  $D \not\in \{5,13,17,21,33,41,65,77\}$,  $1<D<10000$. Then $\Q(\!\sqrt{D})$ admits no ternary universal quadratic lattice.
\end{prop}
\begin{proof}
The proof is mostly computational. Denote $F=\Q(\!\sqrt{D})$. We start by defining two elements of $\co_F^+$: $\alpha_1 := \frac{\lceil \!\sqrt{D} \rceil^{\mathrm{odd}}+\sqrt{D}}{2}$, $\alpha_2 := \lceil \!\sqrt{D} \rceil+ \sqrt{D}$, where $\lceil \,\cdot\,\rceil ^{\mathrm{odd}}$ denotes the smallest odd integer upper bound.
	
Assume first that $D$ is not among the $22$ exceptional numbers $29$, $37$, $53$, $57$, $61$, $69$, $73$, $85$, $89$, $93$, $101$, $105$, $129$, $133$, $161$, $177$, $217$, $253$, $273$, $301$, $329$, $341$. Then one can c          heck that no ternary $\co_F$-lattice represents the following four elements at the same time: $1, 2, \alpha_1, \alpha_2$. This is equivalent to verifying that in $M_4 \bigl(\frac12\co_F\bigr)$, there does not exist a totally positive semidefinite matrix of rank at most 3 with $1, 2, \alpha_1, \alpha_2$ on its diagonal. We checked this by running a simple program in Magma.

Assume now that $D$ is one of the $22$ exceptions, but not $29$, $37$, $57$, $61$, $85$, $105$, $133$, $273$, $329$. Then one can check in Magma that no ternary $\co_F$-lattice represents the following seven elements at the same time: $1, 2, 3, 5, \alpha_1, \alpha_3, \alpha_5$. Here, $\alpha_3 := \frac{\lceil 3\sqrt{D} \rceil^{\mathrm{odd}}+3\sqrt{D}}{2}$ and $\alpha_5 := \frac{\lceil 5\sqrt{D} \rceil^{\mathrm{odd}} + 5\sqrt{D}}{2}$.

Finally, let $D$ be among the nine most problematic values, namely $29$, $37$, $57$, $61$, $85$, $105$, $133$, $273$, $329$. Then, similarly, one can check that no ternary $\co_F$-lattice represents all elements of $S_D$, where:
\begin{align*}
S_{29} &= \{1,2,5,6,\alpha_1,\alpha_1' \},\\
S_{37} &= \{1,2,5,\alpha_1,\alpha_1' \},\\
S_{57} &= \{1,2,4+\omega_D,131+40\omega_D \},\\
S_{61} &= \{1,2,5-\omega_D,9-2\omega_D,92+27\omega_D,133+39\omega_D\},\\
S_{85} &= \{1,2,5+\omega_D,21+5\omega_D,38+9\omega_D\},\\
S_{105} &= \{1,2,5+\omega_D,19+4\omega_D,45-8\omega_D\},\\
S_{133} &= \{1,2,8+\omega_D,27+5\omega_D, 79+15\omega_D\},\\
S_{273} &= \{1,2,8+\omega_D,683+88\omega_D\},\\
S_{329} &= \{1,9+\omega_D,29-3\omega_D,48-5\omega_D,60+7\omega_D\}.
\qedhere
\end{align*}
\end{proof}

And here is the analogous claim in the case when $D\equiv 2,3 \pmod{4}$.

\begin{prop} \label{pr:23mod4}
	Assume that $D \not\equiv 1 \pmod4$ is squarefree, $D \not\in \{2,3,6,7,10\}$, $D<2500$. Then $\Q(\!\sqrt{D})$ admits no ternary universal quadratic lattice.
\end{prop}
\begin{proof}
	Put $F=\Q(\!\sqrt{D})$. Again we start by defining several elements of $\co_F^+$: $\alpha_j := \lceil j\sqrt{D} \rceil+j\sqrt{D}$.
	
	Assume first that $D$ is not among the $26$ exceptional numbers 11, 14, 15, 19, 22, 23, 26, 30, 31, 35, 38, 39, 42, 46, 47, 55, 62, 66, 67, 70, 78, 83, 86, 91, 94, 102. Then, just as in the previous proof, one can check that no ternary $\co_F$-lattice represents the following four elements at the same time: $1, 2, \alpha_1, \alpha_2$.
	
	Assume now that $D$ is one of the exceptions, but not $11$, $14$, $22$, $26$, $38$, $46$, $62$. Then one can check that no ternary $\co_F$-lattice represents the following eight elements at the same time: $1, 2, 3, 5, \alpha_1, \alpha_2', \alpha_3, \alpha_5$.
	
	Finally, let $D$ be among the seven most problematic values, namely $11$, $14$, $22$, $26$, $38$, $46$, $62$. Then, similarly, one can check that no ternary $\co_F$-lattice represents all elements of $S_D$, where:
	\begin{align*}
		S_{11} &= \{ 1,2,6,\alpha_1,\alpha_1+1,\alpha_2,\alpha_3,\alpha_5 \},\\
		S_{14} &= \{ 1,2,5,10,\alpha_1,\alpha_1'+2,35+9\sqrt{14} \},\\
		S_{22} &= \{ 1,5,7,\alpha_1+1,\alpha_5, 197+42\sqrt{22}\},\\ 
		S_{26} &= \{ 1,2,3,\alpha_1,\alpha_1+4,\alpha_5\},\\
		S_{38} &= \{ 1,3,\alpha_1+1,\alpha_3,68+11\sqrt{38} \},\\
		S_{46} &= \{ 1,2,5,\alpha_1,\alpha_3+2\},\\
		S_{62} &= \{ 1,3,5,\alpha_1,63-8\sqrt{62} \}. \qedhere
	\end{align*} 
\end{proof} 

This concludes the proof of Theorem \ref{ThmMain2}. Although it was not needed, we actually ran the computations behind the proofs of both Propositions for all $D \leq 2^{18}=262144$.

For Theorem \ref{thm:ternary}, it only remains to prove that neither $\Q(\!\sqrt{10})$ nor $\Q(\!\sqrt{65})$ admit a free universal ternary lattice. This will follow from Theorem \ref{thm:list}, which provides a complete list of candidates for universal ternary lattices over the remaining fields.

\subsection{Conjectural list of universal ternary lattices} \label{subsec:candidates}

Here we shall give a full (conjectural) list of all universal ternary quadratic lattices over real quadratic fields other than $\Q(\!\sqrt2)$, $\Q(\!\sqrt3)$, $\Q(\!\sqrt5)$. As we established in Theorem \ref{ThmMain2}, such lattices can only exist for fields $\Q(\!\sqrt{D})$ for $D = 6, 7, 10, 13, 17, 21, 33, 41, 65, 77$.

In fields where the fundamental unit is totally positive, we denote it by $\varepsilon_D$, i.e., we have $\varepsilon_6 = 5+2\sqrt6$, $\varepsilon_7 = 8+3\sqrt7$, $\varepsilon_{21}=(5+\sqrt{21})/2$, $\varepsilon_{33}=23+4\sqrt{33}$, $\varepsilon_{77}=(9+\sqrt{77})/2$. We write $Q'$ for the conjugate of the quadratic form $Q$, i.e., the form defined by $Q'(x)=Q(x)'$. For $\Z[\sqrt{10}]$ and $\Z\bigl[\frac{1+\sqrt{65}}2\bigr]$, we use $\mathfrak{p}_2$ to denote a certain dyadic ideal; more precisely, $\mathfrak{p}_2^{-1} = \Z \cdot 1 + \Z \cdot \frac{\sqrt{10}}{2}$ in the former case and $\mathfrak{p}_2^{-1} = \Z \cdot 1 + \Z \cdot \frac{1+\omega_{65}}{2}$ in the latter.

We express the lattices as quadratic forms, since it is more compact than writing Gram matrices. The non-free probably universal lattices over $\Q(\!\sqrt{10})$ and $\Q(\!\sqrt{65})$ are included at the end of the list in the form of a polynomial where the variables $x,y$ take values from $\co_F$ while $z$ can take values from the fractional ideal $\mathfrak{p}_2^{-1}$.

When we say that a free ternary lattice $(L,Q)$ \emph{corresponds} to a quadratic form $\varphi$, we mean that there exists a basis $b_1,b_2,b_3$ of $L$ such that $Q(xb_1+yb_2+zb_3) = \varphi(x,y,z)$. In the non-free case, we similarly have $L = \co_F b_1 + \co_F b_2 + \mathfrak{p}_2^{-1} b_3$ for linearly independent $b_1,b_2,b_3$, and again $Q(xb_1+yb_2+zb_3) = \varphi(x,y,z)$.

We shall need the following quadratic forms: 
\begin{align*}
    \varphi_1^{(6)}(x,y,z) &= x^2+y^2+2z^2 + xz + \sqrt6 yz,\\
    \varphi_2^{(6)}(x,y,z) &= x^2+3y^2+(3+\sqrt6)z^2 + xy + (2+\sqrt6)xz;\\
    \varphi^{(7)}(x,y,z) &= x^2+y^2+2z^2 + \sqrt7 yz;\\
    \varphi_{24}^{(13)}(x,y,z) &= x^2+y^2+2z^2 + \omega_{13} xz + {\omega'_{13}}yz,\\
    \varphi_{12}^{(13)}(x,y,z) &= x^2+y^2+2z^2 + xy + \omega_{13} yz,\\
    \varphi_{\mathrm{h2}}^{(13)}(x,y,z) &= x^2+2y^2+(3+\omega_{13})z^2 + \omega_{13} xy + 2xz,\\
    \varphi_{\mathrm{h3}}^{(13)}(x,y,z) &= x^2+y^2+(2+\omega_{13})z^2 + \omega_{13} yz,\\
    \varphi_{\mathrm{nm}}^{(13)}(x,y,z) &= x^2+2y^2+(2+\omega_{13})z^2 + \omega_{13} xz + \omega_{13} yz;\\
    \varphi_{1}^{(17)}(x,y,z) &= x^2+y^2+2z^2 + xz + \omega_{17} yz,\\
    \varphi_{2}^{(17)}(x,y,z) &= x^2+2y^2+2z^2 + \omega_{17} xy + {\omega'_{17}} xz -yz,\\
    \varphi_3^{(17)}(x,y,z) &= x^2+4y^2+6z^2 + (1-\omega_{17}) xz + (1-4\omega_{17})yz,\\
    \varphi_{\mathrm{nm1}}^{(17)}(x,y,z) &= x^2+y^2+2z^2 + \omega_{17} yz,\\
    \varphi_{\mathrm{nm2}}^{(17)}(x,y,z) &= x^2+(2+\omega_{17})y^2+(6+2\omega_{17})z^2 + (1+\omega_{17}) xy + 3 xz,\\
    \varphi_{\mathrm{nm3}}^{(17)}(x,y,z) &= x^2+(2+\omega_{17})y^2+3z^2 + \omega_{17} xz + 2 yz;\\
    \varphi^{(21)}(x,y,z) &= x^2+y^2+\varepsilon_{21} z^2 + (1-\varepsilon_{21}) yz;\\
    \varphi^{(33)}(x,y,z) &= x^2+y^2+2\varepsilon_{33} z^2 + (7+3\omega_{33})xz + (2+\omega_{33}) yz;\\
    \varphi^{(41)}(x,y,z) &= x^2+4y^2+(15-4\omega_{41}) z^2 + 3xy + (4-\omega_{41}) xz;\\
    \varphi^{(77)}(x,y,z) &= x^2+y^2+2\varepsilon_{77} z^2 + (\varepsilon_{77}-1) yz;\\
    \varphi^{(10)}(x,y,z) &= x^2+(4+\sqrt{10})y^2+4z^2 +2xz + (6+\sqrt{10}) yz;\ x,y\in\co_F,\ z \in \mathfrak{p}_2^{-1};\\
    \varphi^{(65)}(x,y,z) &= x^2+2y^2+4z^2 + xy +2xz + (2-\omega_{65}) yz;\ x,y\in\co_F,\ z \in \mathfrak{p}_2^{-1}.
\end{align*}

The notation is somewhat ad hoc; the important part is that we have a distinct symbol for each form. However, there is some information contained: The upper index identifies the field. In the lower index, numbers $1$, $2$ and $3$ only have a distinguishing function, whereas $12$ and $24$ stand for the number of automorphisms, \enquote{h2} and \enquote{h3} mean that the quadratic form has class number $2$ or $3$, and \enquote{nm} means that the lattice is non-maximal, i.e., that it is contained in another integral lattice.

Now we are able to state a refinement of Theorem \ref{thm:ternary}. Note we will proceed to prove it; the conjectural part is only the fact that all the listed forms are indeed universal.

\begin{thm} \label{thm:list}
Let $F=\Q(\!\sqrt{D})$ be a real quadratic field, $D \neq 2,3,5$ squarefree. Let $L$ be a totally positive definite ternary quadratic $\co_F$-lattice which is universal.

Then $D \in \{6,7,10,13,17,21,33,41,65,77\}$ and furthermore:
 \begin{itemize}
     \item if $D=6$, then $L$ corresponds to $\varphi_1^{(6)}$, $\varepsilon_6 \varphi_1^{(6)}$, $\varphi_2^{(6)}$ or $\varepsilon_6 \varphi_2^{(6)}$;
     \item if $D=7$, then $L$ corresponds to $\varphi^{(7)}$ or $\varepsilon_7 \varphi^{(7)}$;
     \item if $D=10$, then $L$ is non-free and corresponds to $\varphi^{(10)}$;
     \item if $D=13$, then $L$ corresponds to $\varphi_{24}^{(13)}$, $\varphi_{12}^{(13)}$, ${\varphi_{12}^{(13)}}'$, $\varphi_{\mathrm{h2}}^{(13)}$, ${\varphi_{\mathrm{h2}}^{(13)}}'$, $\varphi_{\mathrm{h3}}^{(13)}$, $\varphi_{\mathrm{nm}}^{(13)}$ or ${\varphi_{\mathrm{nm}}^{(13)}}'$;
     \item if $D=17$, then $L$ corresponds to $\varphi_{1}^{(17)}$, ${\varphi_{1}^{(17)}}'$, $\varphi_{2}^{(17)}$, $\varphi_3^{(17)}$, ${\varphi_3^{(17)}}'$, $\varphi_{\mathrm{nm1}}^{(17)}$, ${\varphi_{\mathrm{nm1}}^{(17)}}'$, $\varphi_{\mathrm{nm2}}^{(17)}$, ${\varphi_{\mathrm{nm2}}^{(17)}}'$, $\varphi_{\mathrm{nm3}}^{(17)}$ or ${\varphi_{\mathrm{nm3}}^{(17)}}'$;
     \item if $D=21$, then $L$ corresponds to $\varphi^{(21)}$ or $\varepsilon_{21} \varphi^{(21)}$;
     \item if $D=33$, then $L$ corresponds to $\varphi^{(33)}$ or $\varepsilon_{33} \varphi^{(33)}$;
     \item if $D=41$, then $L$ corresponds to $\varphi^{(41)}$;
     \item if $D=65$, then $L$ is non-free and corresponds to $\varphi^{(65)}$;
     \item if $D=77$, then $L$ corresponds to $\varphi^{(77)}$ or $\varepsilon_{77} \varphi^{(77)}$.
 \end{itemize}
\end{thm}
\begin{proof}
Thanks to Theorem \ref{ThmMain2}, it only remains to show for these $10$ specific fields $F=\Q(\!\sqrt{D})$ that every universal ternary $\co_F$-lattice must be in fact one of the preceding list. We show that every lattice which represents all elements of $\co_F^+$ up to a certain norm is already one of the listed lattices. This is again a computational task, so we only list the sets of elements $S$ such that a lattice $L$ representing all of $S$ is necessarily among the listed ones.

For $D=6$, we checked that no other ternary lattice can at the same time represent $5$ and all elements with norm at most $10$. For $D=7$, no other ternary lattice represents $7+2\sqrt{7}$ as well as all elements with norm at most $9$. For $D=10$, no other ternary lattice represents all elements with norm at most $31$. For $D=13$, no other ternary lattice represents all elements with norm at most $52$. For $D=17$, no other ternary lattice represents all elements with norm at most $30$. For $D=21$, no other ternary lattice represents $5$ as well as all elements with norm at most $15$. For $D=33$, no other ternary lattice represents all elements with norm at most $4$. For $D=41$, no other ternary lattice represents all elements with norm at most $10$. For $D=65$, no other ternary lattice represents all elements with norm at most $16$. For $D=77$, no other ternary lattice represents all elements with norm at most $91$.
\end{proof}

We explicitly state the following conjecture.

\begin{conjecture}
All the $34$ ternary lattices listed in the above theorem are indeed universal.
\end{conjecture}

However, we only have the following:

\begin{thm} \label{thm:h=1}
Each of the $34$ ternary lattices listed in Theorem $\ref{thm:list}$ represents all totally positive definite elements with norm $\leq 250000$. Furthermore, the $5$ lattices over $\Q(\!\sqrt{13})$ and $9$ lattices over $\Q(\!\sqrt{17})$ corresponding to the following quadratic forms are indeed universal:
 \begin{itemize}
     \item $\varphi_{24}^{(13)}$, $\varphi_{12}^{(13)}$, ${\varphi_{12}^{(13)}}'$, $\varphi_{\mathrm{nm}}^{(13)}$ and ${\varphi_{\mathrm{nm}}^{(13)}}'$;
     \item $\varphi_{1}^{(17)}$, ${\varphi_{1}^{(17)}}'$, $\varphi_{2}^{(17)}$, $\varphi_{\mathrm{nm1}}^{(17)}$, ${\varphi_{\mathrm{nm1}}^{(17)}}'$, $\varphi_{\mathrm{nm2}}^{(17)}$, ${\varphi_{\mathrm{nm2}}^{(17)}}'$, $\varphi_{\mathrm{nm3}}^{(17)}$ and ${\varphi_{\mathrm{nm3}}^{(17)}}'$.
 \end{itemize}
\end{thm}
\begin{proof}
For each of the lattices, it is only a computational task to verify that it represents all elements up to certain norm. We used a simple program in Magma to do that.

The $14$ listed lattices all have class number $1$ and one easily verifies that they are locally universal.
\end{proof}

There are several field invariants related to quadratic forms. One of them is the \emph{minimal rank of a universal quadratic $\co_F$-lattice}. However, there are several versions of this invariant. For a totally real number field $F$, we put $m^{\mathrm{free}}(F)$, $m^{\mathrm{free}}_{\mathrm{cl}}(F)$, and $m^{\mathrm{free}}_{\mathrm{diag}}(F)$ to be the minimal rank of a possibly non-classical / classical / diagonal universal quadratic form, respectively. Similarly, we also write $m(F)$ and $m_{\mathrm{cl}}(F)$ for the two versions of the minimal rank of a universal quadratic lattice. Although it always seemed probable that these values are in general not the same, their exact values are known so rarely that, until now, it was theoretically possible to conjecture that $m^{\mathrm{free}} = m^{\mathrm{free}}_{\mathrm{cl}} = m = m_{\mathrm{cl}}$ for every field. (Note that, e.g., for $\Q(\!\sqrt{D})$, $D=2,3,5$, all five invariants coincide by the results of \cite{CKR}, all being equal to $3$; and a combination of \cite{KKP} and \cite{Ki2} yields that there are infinitely many real quadratic fields where all of them are $8$.) We conclude the paper by two remarks about strict inequalities between these invariants.

\begin{remark}\label{rem1}
For $F = \Q(\!\sqrt{13})$, Sasaki \cite{Sa} proved that there are precisely two classical universal quaternary forms over $\co_F$; no classical universal ternary form exists \cite{CKR}. This shows $m^{\mathrm{free}}_{\mathrm{cl}}(F)=4$. In Theorem \ref{thm:h=1}, we have found multiple non-classical universal ternary forms, hence $m^{\mathrm{free}}(F)=3$. It is easy to check that neither of the two Sasaki's forms is diagonalizable. Hence, a diagonal universal form requires at least $5$ variables, which means $m^{\mathrm{free}}_{\mathrm{diag}}(F) \geq 5$. (In fact, there are several quinary diagonal forms which seem to be universal, e.g.\ $v^2+w^2+x^2+\frac{5+\sqrt{13}}2 y^2 + \frac{5-\sqrt{13}}{2} z^2$.)
\end{remark}

We have also produced two explicit conjectural examples of real quadratic fields where the minimal rank $m$ of a universal lattice is strictly smaller than the minimal rank $m^{\mathrm{free}}$ of a free universal quadratic lattice. 

\begin{remark}\label{rem2}
Let $F$ be either $\Q(\!\sqrt{10})$ or $\Q(\!\sqrt{65})$. We have shown that there is exactly one candidate for a universal ternary $\co_F$-lattice (see Theorem \ref{thm:list}), and this lattice is non-free. It can also be observed to be non-classical. Thus we know that $m^{\mathrm{free}}(F)>3$ and $m_{\mathrm{cl}}(F)>3$, while we conjecture $m(F)=3$.
In fact, every non-free lattice of rank $r$ is contained in a free lattice of rank $r+1$, so the conjecture would also imply $m^{\mathrm{free}}(F)=4$.
\end{remark}

As a corollary of Theorem \ref{thm:ternary}, we see that all real quadratic fields admitting a \emph{free} universal ternary lattice have class number $1$. The only two fields of higher degree where a universal ternary lattice has been found (Krásenský--Scharlau [in preparation]), namely $\Q(\!\sqrt2,\sqrt5)$ and $\Q(\zeta_{20})^+$, both have class number $1$ as well. This leads us to the following question, related to Kitaoka's Conjecture:

\begin{quest}
Let $F$ be a totally real number field admitting a free universal ternary lattice. Does it follow that the class number of $F$ is $1$?
\end{quest}

\subsection*{Acknowledgments} We are very grateful to Markus Kirschmer for sharing a Magma code which checks universality of a quadratic lattice up to a given norm.


\begin{thebibliography}{HHX}

\bibitem[An]{An} N. C. Ankeny, \emph{The least quadratic non residue}, Ann. of Math. \textbf{55}, no. 1 (1952), 65–72.

\bibitem[BK]{BK1} V. Blomer,  V. Kala, {\em Number fields without universal $n$-ary quadratic forms}, Math. Proc. Cambridge Philos. Soc. \textbf{159} (2015), 239--252.

\bibitem[Ber]{Be} P. Bernays, \textit{\" Uber die Darstellung von positiven, ganzen Zahlen durch die primitiven
bin\" aren quadratischen Formen einer nichtquadratischen Diskriminante}, Dissertation, G\" ottingen, 1912.

\bibitem[Ben]{Ben} A. A. Bennett, \emph{Table of Quadratic Residues}, Ann. of Math., \textbf{27}(4) (1926), 349--356.

\bibitem[Bh]{Bh} M. Bhargava, \emph{On the Conway-Schneeberger fifteen theorem}, Contemp. Math, \textbf{272} (1999), 27--37.

\bibitem[BH]{BH} M. Bhargava, J. Hanke, \emph{Universal quadratic forms and the 290-theorem}, 2011. Preprint.

\bibitem[BC]{BC} M. Bordignon, G. Cherubini, \textit{Coprime-universal quadratic forms}, arxiv:2406.01533.

	\bibitem[BCP]{magma} W. Bosma, J. Cannon, C. Playoust, \emph{The Magma algebra system. I. The user language}, J. Symbolic Comput. \textbf{24} (1997), 235--265.

\bibitem[Bu]{Bu} D. A. Burgess, \emph{The distribution of quadratic residues and non-residues}, Mathematika, \textbf{4} (1957), 106--112.

\bibitem[CKR]{CKR} W. K. Chan, M.-H. Kim, S. Raghavan, \emph{Ternary universal integral quadratic forms}, Japan. J. Math. \textbf{22} (1996), 263--273.


\bibitem[CO]{CO} W. K. Chan, B.-K. Oh, {\em Can we recover an integeral quadratic form by representing all its subforms?}, Adv. Math. \textbf{433} (2023), Paper No.\ 109317.

\bibitem[Ch]{Chau} S. S. Chaudhury, \textit{Sums of squares of integral multiples of an integral element of real bi-quadratic fields}, Ramanujan J. \textbf{64} (2024), 419--451.


\bibitem[DSP]{DSP} W. Duke, R. Schulze-Pillot, \textit{Representation of integers by positive ternary quadratic
forms and equidistribution of lattice points on ellipsoids}, Invent. Math. \textbf{99}
(1990), 49--57.

\bibitem[Ea]{Ea} A. G. Earnest,\emph{Universal and regular positive quadratic lattices over totally real number fields}, Contemp. Math. \textbf{249}  (1999), 17--27.


\bibitem[EK]{EK} A. G. Earnest, A. Khosravani, {\em Universal positive quaternary quadratic lattices over totally real number fields}, Mathematika \textbf{44} (1997), 342--347.


\bibitem[Goe]{Go} F. G\" otzky, \textit{\" Uber eine zahlentheoretische Anwendung von Modulfunktionen zweier Ver\" anderlicher}, Math. Ann. \textbf{100} (1928), 411--437.

\bibitem[HHX]{HHX} Z. He, Y. Hu, F. Xu, \emph{On indefinite {$k$}-universal integral quadratic forms over number fields}, Math. Z. \textbf{304} (2023), no.1, Paper No. 20, 26 pp. 


\bibitem[HJ]{HJ} R. A. Horn, C. R. Johnson, {\em Matrix analysis}, 
Second edition, Cambridge University Press, Cambridge, 2013.

	\bibitem[Ka1]{Ka1} V. Kala, {\em Universal quadratic forms and elements of small norm in real quadratic fields}, Bull. Aust. Math. Soc. \textbf{94} (2016), 7--14.
	
	\bibitem[Ka2]{Ka2} V. Kala, {\em Universal quadratic forms and indecomposables in number fields: A survey}, Commun. Math. \textbf{31} (2023), 81--114.

    \bibitem[Ka3]{Ka3} V. Kala, \textit{Number fields without universal quadratic forms of small rank exist in most degrees}, Math. Proc. Cambridge Philos. Soc. \textbf{174} (2023), 225--231.
    
	\bibitem[KT]{KT} V. Kala,  M. Tinkov\'a, {\em Universal quadratic forms, small norms and traces in families of number fields}, Int. Math. Res. Not. IMRN (2023), 7541--7577.

\bibitem[KY1]{KY2} V. Kala, P. Yatsyna, \textit{Sums of squares in S-integers}, New York J. Math. \textbf{26} (2020), 1145--1154.

	\bibitem[KY2]{KY1} V. Kala,  P. Yatsyna, {\em Lifting problem for universal quadratic forms}, Adv. Math. \textbf{377} (2021), 107497.

\bibitem[KY3]{KY} V. Kala, P. Yatsyna, {\em On {K}itaoka's conjecture and lifting problem for universal quadratic forms}, Bull. Lond. Math. Soc., \textbf{55}(2) (2023), 854--864. 



	\bibitem[KYZ]{KYZ} V. Kala, P. Yatsyna,  B. \.Zmija, {\em Real quadratic fields with a universal form of given rank have density zero},  arXiv:2302.12080.


	\bibitem[Ki1]{Ki1} B. M. Kim, {\em Finiteness of real quadratic fields which admit positive integral diagonal septenary universal forms}, Manuscr. Math. \textbf{99} (1999), 181--184.
	
	\bibitem[Ki2]{Ki2} B. M. Kim, {\em Universal octonary diagonal forms over some real quadratic fields}, Comment. Math. Helv. \textbf{75} (2000), 410--414. 

\bibitem[Ki3]{kim-unp} B. M. Kim, \emph{Positive universal forms over totally real fields}, 2003. Unpublished manuscript.

	
\bibitem[KKP]{KKP} B. M. Kim, M.-H. Kim,  D. Park, {\em Real quadratic fields admitting universal lattice of rank $7$}, J. Number Theory \textbf{233} (2022), 456--466.
  
    \bibitem[KL]{KL} D. Kim, S. H. Lee, \textit{Lifting problem for universal quadratic forms over totally real cubic number fields},  Bull. Lond. Math. Soc. \textbf{56} (2024), 1192--1206.

\bibitem[Kim]{Ki} M.-H. Kim, \emph{Recent developments on universal forms}, Contemp. Math. \textbf{344} (2004), 215--228.

\bibitem[Kir]{Kir} S. Kirmse, \textit{Zur Darstellung total positiver Zahlen als Summen von vier Quadraten}, Math. Z. \textbf{21} (1924), 195--202.
	
	\bibitem[KTZ]{KTZ} J. Kr{\'a}sensk{\'y}, M. Tinkov{\'a},  K. Zemkov{\'a}, {\em There are no universal ternary quadratic forms over biquadratic fields}, Proc. Edinb. Math. Soc. \textbf{63} (2020), 861--912.

\bibitem[LLS]{LLS} Y. Lamzouri, X. Li, K. Soundararajan, \emph{Conditional bounds for the least quadratic non-residue and related problems}, Math. Comp. \textit{84} (2015), 2391--2412.

\bibitem[Ma]{M} H. Maass, \emph{\"{U}ber die {D}arstellung total positiver {Z}ahlen des {K}\"{o}rpers {$R(\!\sqrt5)$} als {S}umme von drei {Q}uadraten}, Abh. Math. Sem. Hamburg \textbf{14} (1941), 185--191.


    \bibitem[Man]{Man} S. H. Man, \textit{Minimal rank of universal lattices and number of indecomposable elements in real multiquadratic fields}, Adv. Math. \textbf{447} (2024), 109694.

    \bibitem[Mar]{Mar} M. Marcus, \textit{Determinants of Sums}, College Math. J. 21 (1990), no. 2, 130--135.

\bibitem [OM]{OM} O. T. O'Meara, {\em Introduction to quadratic forms}, Springer-Verlag, New York, 1963.

\bibitem[Ras]{Ras} M. Ra\v ska, \textit{Representing multiples of $m$ in real quadratic fields as sums of squares}, J. Number Theory \textbf{244} (2023), 24--41.

\bibitem[Sa]{Sa} H. Sasaki, \emph{Quaternary universal forms over ${\Q}(\!\sqrt{13})$}, Ramanujan J.
 \textbf{18} (2009), 73--80.

 
\bibitem[Si]{Si} C. L. Siegel, {\em Sums of $m$-th powers of algebraic integers}, Ann. of Math. \textbf{46} (1945), 313--339.

\bibitem[Tre]{Tre} E. Trevi\~{n}o, \textit{The least k-th power non-residue}, J. Number Theory \textbf{149} (2015), 201--224.

    \bibitem[XZ]{XZ} F. Xu, Y. Zhang, \textit{On indefinite and potentially universal quadratic forms over number fields}, Trans. Amer. Math. Soc. \textbf{375} (2022), 2459--2480.

	\bibitem[Ya]{Ya} P. Yatsyna, {\em A lower bound for the rank of a universal quadratic form with integer coefficients in a totally real field}, Comment. Math. Helv. \textbf{94} (2019), 221--239.

\end{thebibliography}
\end{document}